%% file: main.tex
\documentclass[ijoc,sglanonrev]{informs4}
\usepackage{eqndefns-left} \RequirePackage{tgtermes}
\RequirePackage{newtxtext}
\RequirePackage{newtxmath}
\RequirePackage{bm}
\RequirePackage{endnotes}

\OneAndAHalfSpacedXII 
\usepackage{algorithm}
\usepackage{algpseudocode}
\usepackage{tikz}
 \usepackage{booktabs}
 \usepackage{booktabs}
\usepackage{longtable}
\renewenvironment{proof}[1][Proof]
{\par\noindent\textit{#1.}\ }
{\hfill$\square$\par}

\usepackage{natbib}
 \bibpunct[, ]{(}{)}{,}{a}{}{,}\def\bibfont{\small}\usepackage{siunitx}

\usepackage{multirow}
\usepackage{tabularx}
\usepackage{overpic}
\usepackage[caption=false]{subfig}
\usepackage{comment}

\usepackage{subcaption}
\usepackage[justification=centering]{caption}

\EquationsNumberedThrough    
\TheoremsNumberedThrough     \ECRepeatTheorems  
\MANUSCRIPTNO{IJOC-0001-2024.00}

\makeatletter

\renewcommand{\theARTICLETOP}{}

\LRHSecondLine{}
\RRHSecondLine{}

\renewcommand{\theARTICLEABSTRACT}{\par
  \vspace{12pt}\noindent
  \begin{minipage}{\textwidth}
    \parindent1em
    \ABSfont
    \noindent\theABSTRACT\par
    \vskip5pt
    \theFUNDING
    \theKEYWORDS
    \theSUBJECTCLASS
    \theAREAOFREVIEW
    \theMSCCLASS
    \theORMSCLASS
    \par\noindent\hrulefill
  \end{minipage}\par
}

\makeatother

\begin{document}

\RUNAUTHOR{Constante Flores et al.}

\RUNTITLE{Self-Supervised Tuning of Low-Fidelity Optimization Models}

\TITLE{Self-Supervised Tuning of Low-Fidelity Optimization Models via High-Fidelity Feedback}

\ARTICLEAUTHORS{\AUTHOR{Gonzalo E. {Constante Flores}}
\AFF{Department of Electrical, Computer \& Energy Engineering,
University of Colorado Boulder, \EMAIL{gonzalo.constante@colorado.edu}}

\AUTHOR{André Quisaguano}
\AFF{Department of Electrical, Computer \& Energy Engineering,
University of Colorado Boulder, \EMAIL{andre.quisaguanoparedes@colorado.edu}}
\AUTHOR{Asha Ramanujam}
\AFF{Davidson School of Chemical Engineering,
Purdue University, \EMAIL{aramanuj@purdue.edu}}
\AUTHOR{Rishabh Gupta}
\AFF{ExxonMobil Technology and Engineering Company, \EMAIL{rishabh.gupta1@exxonmobil.com}}
\AUTHOR{Can Li}
\AFF{Davidson School of Chemical Engineering,
Purdue University, \EMAIL{canli@purdue.edu}}
} 
\ABSTRACT{Low-fidelity (LF) optimization models are widely used in large-scale operations because they can be solved efficiently and reliably, but they often exhibit systematic errors relative to high-fidelity (HF) formulations that better capture physical and operational constraints. Existing learning-based approaches typically rely on supervised or semi-supervised training with HF solutions, which are expensive to compute and may not align with HF operational objectives. We propose a self-supervised framework that parameterizes LF optimization models directly from HF feedback. The low-fidelity problem is embedded as a differentiable optimization layer and trained end-to-end, while a HF correction problem evaluates its solution and provides gradient information through its dual variables, without requiring HF solution labels. We evaluate the framework on optimal power flow, optimal gas flow, and pooling problems, spanning different sources of low- and high-fidelity mismatch. Across all three applications, the self-supervised approach improves objective quality and feasibility over untuned LF models and generally outperforms semi-supervised training. The learned parameterizations also generalize to unseen operating conditions and, for network applications with up to 30,000 nodes, to systems not encountered during training.
}


\maketitle

\section{Introduction}\label{sec:Intro}

High-fidelity (HF) optimization models are essential for decision-making in many engineering applications because they accurately capture complex physical, economic, and operational constraints. However, their repeated solution is often computationally prohibitive due to large-scale nonconvex formulations. As a result, practitioners commonly rely on low-fidelity (LF) models that trade modeling accuracy for computational efficiency through approximations, relaxations, or simplified physics. A representative example is power systems, where the alternating current optimal power flow (AC-OPF) model \citep{frank2016introduction} accurately captures nonlinear power flow equations, voltage magnitudes, reactive power, and transmission losses but is computationally expensive. In contrast, the direct current optimal power flow (DC-OPF) model simplifies these equations by assuming fixed voltage magnitudes, negligible line losses, and small voltage angle differences, yielding a much more scalable model at the expense of physical accuracy.

Although LF models are computationally attractive, they often exhibit systematic mismatches relative to their HF counterparts, which may arise from linearizations, omitted constraints, surrogate objective functions, or simplified physical assumptions, and can lead to suboptimal or even infeasible decisions. Previous work by \citet{Constante-Flores2025,Gupta2024} showed that the parameters of a nominal LF model can be systematically calibrated to better reproduce HF system behavior while preserving the LF model's physical interpretability and computational efficiency.

Figure~\ref{fig:overview} illustrates the geometric intuition behind this idea. The shaded regions represent the feasible sets of the HF and LF models, where \(x^\star\) denotes the HF-optimal solution and \(\hat{x}\) denotes the solution obtained from the nominal LF model. Because the LF model relies on simplified physics and approximations, its solution may deviate substantially from \(x^\star\). Rather than replacing the LF model with an unstructured surrogate, the calibration approach adjusts a set of physically meaningful LF model parameters, effectively shifting or reshaping the solution geometry so that the resulting LF decisions better approximate those of the HF model. However, existing approaches require repeatedly solving computationally expensive bilevel optimization problems, with the LF optimization problem embedded in the lower level, limiting their scalability to large-scale problems.

\paragraph{Contributions} To overcome this limitation, we propose a scalable self-supervised learning framework for tuning LF optimization models using HF operational feedback. Specifically, the LF optimization problem is embedded as a differentiable optimization layer whose parameters are trained end-to-end to minimize the HF objective induced by the LF decisions. Instead of requiring HF-optimal solutions as supervision, the HF model acts as an evaluation layer that computes the operational adjustment needed to satisfy the HF constraints and returns the resulting objective value together with sensitivity information for gradient-based training. As a result, the framework directly optimizes LF model parameters with respect to HF performance while avoiding the costly bilevel optimization procedures required by previous approaches. We further characterize the structural approximation capacity of the parameterized LF model through its reachable decision set and local solution sensitivities, which clarifies when matching an HF reference decision and minimizing HF operational loss can lead to different LF decisions.

\begin{figure}
    \centering
    \includegraphics[width=0.5\linewidth]{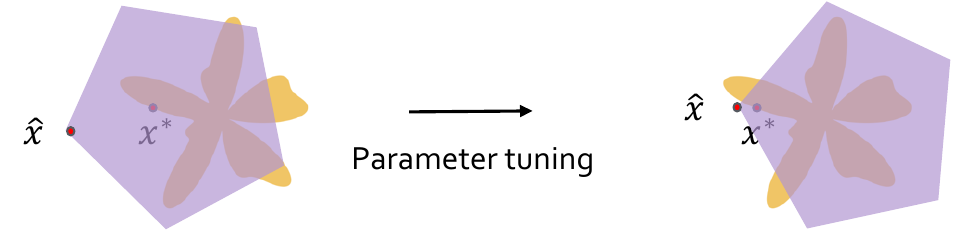}
    \caption{Geometric intuition of tuning low fidelity surrogate to better approximate a high-fidelity model.}
    \label{fig:overview}
\end{figure}

We evaluate the framework on three classes of optimization problems that represent qualitatively different sources of LF--HF mismatch: (i) DC approximations of nonlinear AC power-flow physics, (ii) relaxations of nonlinear gas-flow equations, and (iii) McCormick-based approximations of bilinear constraints in pooling problems. This diversity allows us to evaluate whether the proposed learning principle extends beyond a particular application or approximation structure. Across all three applications, the proposed framework consistently outperforms untuned LF models and semi-supervised baselines in terms of HF objective value, feasibility, computational scalability, and generalization to unseen operating conditions. The numerical experiments demonstrate that self-supervised learning driven by HF optimization feedback provides a principled and scalable framework for improving tractable surrogate optimization models for large-scale decision-making.

\paragraph{Related Work}
This paper lies at the intersection of contextual optimization, decision-focused learning (DFL), and optimization-based surrogate modeling. Predictive-to-prescriptive methods use context to learn decisions or problem data that minimize downstream cost rather than prediction error \citep{Bertsimas2020,Elmachtoub2022}. DFL further aligns training with downstream decision quality through implicit sensitivities, convex surrogates, continuous relaxations, perturbation-based gradients, black-box differentiation, and related loss constructions \citep{Donti2017,Wilder2019,Berthet2020,Vlastelica2020,Mulamba2021,Mandi2022,HoNguyenKilincKarzan2022,Kotary2021,Mandi2024a,Mandi2026}. Our formulation follows this task-alignment principle but targets a different source of misspecification. Rather than estimating unknown data for a prescribed optimization model, we learn context-dependent corrections to a structured LF model and evaluate its decisions with a distinct HF model. The framework therefore addresses model-form error across fidelity levels rather than only statistical error in problem data.

Differentiable optimization layers enable end-to-end training through optimization problems. OptNet differentiates quadratic programs through their KKT systems \citep{Amos2017}, while disciplined-parameterized and conic formulations extend this capability to broader classes of convex programs \citep{Agrawal2019}. More general frameworks support solver-agnostic implicit differentiation \citep{Blondel2022} and derivatives with respect to model parameters in algebraic modeling systems \citep{Besancon2024}. We use these tools to differentiate the LF solution map and propagate HF operational sensitivities back to the parameters of the LF model.

The closest work learns or calibrates structured optimization models rather than replacing them with direct decision predictors. \citet{Agrawal2021} introduce convex optimization models whose parameters are learned from labeled data by differentiating through the solution map. \citet{Ferber2023} propose SurCo, which learns the objective coefficients of a linear combinatorial surrogate end-to-end so that its solution minimizes a differentiable nonlinear objective, without requiring optimal-solution labels. Unlike our setting, however, SurCo retains the original feasible region and focuses on replacing a nonlinear objective with a learned linear surrogate, assuming direct access to the nonlinear objective and its gradient. \citet{Gupta2024} learn simpler optimization models from HF-optimal decisions through bilevel inverse optimization, while \citet{Constante-Flores2025} calibrate and then learn parameters of a DC-OPF approximation using AC-OPF information in a supervised setting. Related work in power systems has also sought to improve tractable OPF approximations. \citet{TaheriMolzahn2026} optimize parameters of the DC-OPF offline over a collection of operating conditions, yielding a single calibrated approximation used across instances. Our framework extends beyond power flow, allows selected LF parameters and contextual inputs to be tuned, and replaces offline label generation or direct objective gradients with end-to-end feedback from a distinct HF optimization problem.

\section{Proposed Framework}\label{sec:Framework}
We now formalize the proposed framework. We begin by presenting generic formulations of the HF and LF optimization models, followed by supervised, semi-supervised, and self-supervised learning approaches for tuning the LF model. We conclude the section with a qualitative discussion explaining why self-supervised learning is particularly well suited to this problem.
\subsection{Problem Setup}

\paragraph{Low-fidelity optimization model.}
The low-fidelity model provides a nominal decision policy based on a simplified or approximate representation of the underlying system. Let $x_{\mathrm{lf}} \in \mathbb{R}^{n_{\mathrm{lf}}}$ denote the LF decision variables, let $\rho_{\mathrm{lf}}\in \mathbb{R}^{d_{\rho}}$ denote the fixed parameters of the LF model, and let $\xi \in \mathbb{R}^{k}$ denote the changing operating environment, such as the changing demand in an optimal power flow problem. We assume that
$
    \xi \sim \mathbb{P}_{\xi},
$
and that i.i.d. samples can be generated from $\mathbb{P}_{\xi}$.

For a given realization of $\xi$, the nominal LF decision is obtained by solving
\begin{equation}
\label{eq:lf_nominal}
x_{\mathrm{lf}}(\xi)
\in
\argmin_{x_{\mathrm{lf}}}
\left\{
f(x_{\mathrm{lf}};\rho_{\mathrm{lf}},\xi)
\;\middle|\;
g(x_{\mathrm{lf}};\rho_{\mathrm{lf}},\xi)\leq 0,\;
h(x_{\mathrm{lf}};\rho_{\mathrm{lf}},\xi)=0
\right\}.
\end{equation}

Here, $f$, $g$, and $h$ denote the objective, inequality-constraint, and equality-constraint functions of the LF model, respectively. We assume that $f$ and each component of $g$ are convex in $x_{\mathrm{lf}}$, while $h$ is affine in $x_{\mathrm{lf}}$. The parameter vector $\rho_{\mathrm{lf}}$ may contain fixed quantities such as network configurations, physical properties, or coefficients introduced through modeling approximations that do not change with operating conditions.

\paragraph{High-fidelity optimization model.}
The high-fidelity model provides a more detailed representation of the physical or operational behavior of the system. Let $x_{\mathrm{hf}} \in \mathbb{R}^{n_{\mathrm{hf}}}$ denote the HF decision variables. For a given $\xi$, the nominal HF problem is
\begin{equation}
\label{eq:hf_nominal}
x_{\mathrm{hf}}(\xi)
\in
\argmin_{x_{\mathrm{hf}}}
\left\{
F(x_{\mathrm{hf}};\rho_{\mathrm{hf}},\xi)
\;\middle|\;
G(x_{\mathrm{hf}};\rho_{\mathrm{hf}},\xi)\leq 0,\;
H(x_{\mathrm{hf}};\rho_{\mathrm{hf}},\xi)=0
\right\}.
\end{equation}
The parameter vector $\rho_{\mathrm{hf}}$ contains the fixed parameters of the HF model. We assume that $F$, $G$, and $H$ are continuously differentiable with respect to $x_{\mathrm{hf}}$, but they may be nonconvex.

\paragraph{Relationship between the LF and HF decisions.}
The LF and HF models need not have identical decision spaces. We assume, however, that they contain a set of semantically compatible decision components. Let
$$
S_{\mathrm{lf}}
\in \{0,1\}^{q\times n_{\mathrm{lf}}}
\quad\text{and}\quad
S_{\mathrm{hf}}
\in \{0,1\}^{q\times n_{\mathrm{hf}}}
$$
denote selection matrices that extract these shared components from the LF and HF decision vectors, respectively. $S_{\mathrm{lf}}x_{\mathrm{lf}}$ and $S_{\mathrm{hf}}x_{\mathrm{hf}}$ represent the variables with the same dimension and semantic meaning.  The LF and HF feasible sets need not coincide when projected onto the shared decision components. An LF solution therefore defines only a partial decision for the HF evaluation problem, where it is imposed or coupled to the corresponding HF variables. Slack variables in the HF evaluation problem quantify any resulting incompatibility with the HF model.

Given a realization of $\xi$ and the associated LF decision $x_{\mathrm{lf}}$, we evaluate the LF decision using the following slack-augmented HF problem:
\begin{subequations}\label{eq:hf_problem_eval}
\begin{align}
    \bigl(
        x_{\mathrm{hf}}^\star(\xi,x_{\mathrm{lf}}),
        \delta_+^\star(\xi,x_{\mathrm{lf}}),
        \delta_-^\star(\xi,x_{\mathrm{lf}})
    \bigr)
    \in
    \argmin_{x_{\mathrm{hf}},\,\delta_+,\,\delta_-}
    \quad &
    F\bigl(x_{\mathrm{hf}};\rho_{\mathrm{hf}},\xi\bigr)
    +\Psi(\delta_+,\delta_-)
    \label{eq:hf_problem_eval_a}\\
    \text{s.t.}\quad
    &
    G\bigl(x_{\mathrm{hf}}, \delta_{-};\; x_{\mathrm{lf}}, \rho_{\mathrm{hf}},\xi\bigr)
    \leq 0,
    \label{eq:hf_problem_eval_b}\\
    &
    H\bigl(x_{\mathrm{hf}},\delta_+, \delta_-;\;x_{\mathrm{lf}},\rho_{\mathrm{hf}},\xi\bigr)
    = 0,
    \label{eq:hf_problem_eval_c}\\
    &
    \delta_+\geq 0,\qquad \delta_-\geq 0.
    \label{eq:hf_problem_eval_d}
\end{align}
\end{subequations}
Here, the dependence of $G$ and $H$ on $x_{\mathrm{lf}}$ represents the coupling between the LF decision and the HF evaluation problem. Only the shared components $S_{\mathrm{lf}}x_{\mathrm{lf}}$ are passed to the HF model. Additional constraints may be added to couple the HF and LF decisions with slack variables, i.e., $S_{\mathrm{lf}}x_{\mathrm{lf}} = S_{\mathrm{hf}}x_{\mathrm{hf}}+\delta_{+} - \delta_{-}$. With a slight abuse of notation, the dependence of $G$ and $H$ on $\delta_+$ and $\delta_-$ indicates that slack variables may be introduced into selected HF constraints, including constraints used to couple the LF and HF decisions. This general formulation allows the particular location and form of the feasibility corrections to depend on the application.

We assume that $\Psi:\mathbb{R}_+^{n+}\times\mathbb{R}_+^{n-}\rightarrow\mathbb{R}_+$ is convex, satisfies $\Psi(0,0)=0$, and is nondecreasing in the magnitude of the slack variables. For example, $\Psi(\delta_{+},\delta_{-})=\lambda_{+}^{\top}\delta_{+}+\lambda_{-}^{\top}\delta_{-}$ with $\lambda_{+},\lambda_{-}\ge 0$.

The slack-augmented $G$ and $H$ are assumed to be continuously differentiable with respect to $x_{\mathrm{hf}}$, $\delta_+$, and $\delta_-$. We define the HF evaluation loss $\ell_{\mathrm{hf}}(x_{\mathrm{lf}},\xi)$ as the optimal value of \eqref{eq:hf_problem_eval}. This loss accounts for both the operational performance achieved by the HF model and the corrections required to reconcile the LF and HF decisions.

\paragraph{Tuning the LF model.}
The proposed framework tunes selected parameters of the LF model so that the resulting LF decisions perform well when evaluated using \eqref{eq:hf_problem_eval}. Let
$
\phi_\theta:\mathbb{R}^k\rightarrow\mathbb{R}^p
$
be a parametric mapping with trainable parameters $\theta$. We define the tuned fixed and operating parameters as
\begin{equation} \label{eq:tuned_parameters}
    \rho_{\theta}^{\mathrm{lf}}(\xi) =
    \rho_{\mathrm{lf}}+A\phi_\theta(\xi),
    \quad \text{and} \quad
    \xi_{\theta}(\xi) =
    \xi+B\phi_\theta(\xi),
\end{equation}
where
$A\in\mathbb{R}^{d_\rho\times p}$ and $B\in\mathbb{R}^{k\times p}$ are fixed matrices of appropriate dimensions. The matrices $A$ and $B$ are typically sparse so that only a selected subset of the LF parameters are modified.

The selection of tunable parameters is application dependent and is generally guided by physical insight. Jointly learning which parameters should be tuned would lead to a best-subset selection problem and may substantially increase the computational cost. We therefore assume that the tunable components are specified in advance.

The resulting tuned LF problem is
\begin{equation}
\label{eq:lf_nominal_tuned}
x_{\mathrm{lf}}^\theta(\xi)
\in
\argmin_{x_{\mathrm{lf}}}
\left\{
f\!\left(
x_{\mathrm{lf}};
\rho_{\theta}^{\mathrm{lf}}(\xi),
\xi_{\theta}(\xi)
\right)
\;\middle|\;
g\!\left(
x_{\mathrm{lf}};
\rho_{\theta}^{\mathrm{lf}}(\xi),
\xi_{\theta}(\xi)
\right)\leq 0,\;
h\!\left(
x_{\mathrm{lf}};
\rho_{\theta}^{\mathrm{lf}}(\xi),
\xi_{\theta}(\xi)
\right)=0
\right\}.
\end{equation}

The notation $x_{\mathrm{lf}}^\theta(\xi)$ emphasizes that the LF decision depends on both $\xi$ and the learned parameters $\theta$.

\paragraph{Expected and empirical learning objectives.}
The expected HF loss induced by the tuned LF model is
\begin{equation}
    \mathcal{L}(\theta)
    :=
    \mathbb{E}_{\xi\sim\mathbb{P}_{\xi}}
    \left[
        \ell_{\mathrm{hf}}
        \bigl(x_{\mathrm{lf}}^\theta(\xi),\xi\bigr)
    \right].
    \label{eq:expected_loss}
\end{equation}

Suppose that we can draw samples
$
    \{\xi_i\}_{i=1}^n
    \overset{\mathrm{i.i.d.}}{\sim}
    \mathbb{P}_{\xi}.
$
The corresponding empirical risk minimization problem can be formulated as the following bilevel optimization problem:
\begin{subequations}\label{eq:empirical_bilevel}
\begin{align}
    \min_{\substack{
        \theta,\,
        \{x_{\mathrm{lf},i},x_{\mathrm{hf},i},
        \delta_{+,i},\delta_{-,i}\}_{i=1}^{n}
    }}
    \quad &
    \frac{1}{n}
    \sum_{i=1}^{n}
    \left[
        F\bigl(
            x_{\mathrm{hf},i};
            \rho_{\mathrm{hf}},\xi_i
        \bigr)
        +
        \Psi\bigl(\delta_{+,i},\delta_{-,i}\bigr)
    \right]
    \label{eq:empirical_bilevel_a}
    \\
    \text{s.t.}\quad
    &
    G\bigl(
        x_{\mathrm{hf},i},
        x_{\mathrm{lf},i},
        \delta_{-,i};
        \rho_{\mathrm{hf}},\xi_i
    \bigr)
    \leq 0,
    \qquad i=1,\ldots,n,
    \label{eq:empirical_bilevel_b}
    \\
    &
    H\bigl(
        x_{\mathrm{hf},i},
        x_{\mathrm{lf},i},
        \delta_{+,i},
        \delta_{-,i};
        \rho_{\mathrm{hf}},\xi_i
    \bigr)
    = 0,
    \qquad i=1,\ldots,n,
    \label{eq:empirical_bilevel_c}
    \\
    &
    \delta_{+,i}\geq 0,\qquad
    \delta_{-,i}\geq 0,
    \qquad i=1,\ldots,n,
    \label{eq:empirical_bilevel_d}
    \\
    &
    x_{\mathrm{lf},i}
    \in
    \argmin_{x_{\mathrm{lf}}}
    \left\{
    \begin{aligned}
        \quad &
        f\bigl(
            x_{\mathrm{lf}};
            \rho_{\theta}^{\mathrm{lf}}(\xi_i),
            \xi_{\theta}(\xi_i)
        \bigr)
        \\
        \text{s.t.}\quad &
        g\bigl(
            x_{\mathrm{lf}};
            \rho_{\theta}^{\mathrm{lf}}(\xi_i),
            \xi_{\theta}(\xi_i)
        \bigr)
        \leq 0,
        \\
        &
        h\bigl(
            x_{\mathrm{lf}};
            \rho_{\theta}^{\mathrm{lf}}(\xi_i),
            \xi_{\theta}(\xi_i)
        \bigr)
        = 0
    \end{aligned}
    \right\},
    \qquad i=1,\ldots,n.
    \label{eq:empirical_bilevel_e}
\end{align}
\end{subequations}
where the upper-level objective evaluates the decisions produced by the tuned LF model using the HF model. For fixed LF decisions, the variables $x_{\mathrm{hf},i}$, $\delta_{+,i}$, and $\delta_{-,i}$ minimize the HF objective and correction penalty subject to the HF constraints. The lower-level problems in \eqref{eq:empirical_bilevel_e} ensure that each $x_{\mathrm{lf},i}$ is an optimal solution of the corresponding tuned LF problem. Training the machine learning model $\phi_{\theta}$ by solving \eqref{eq:empirical_bilevel} directly is intractable. In subsections \ref{sec:supervised}-\ref{sec:self_supervised}, we present several learning algorithms for approximately solving the learning problem.

\subsection{Approximation Capacity of the Parameterized LF Model}
\label{subsec:lf_approximation_capacity}

We distinguish the expressive capacity of the parameterized LF model from the ability of a learning method to identify suitable parameter corrections. The former determines which decisions can be induced by modifying the selected LF parameters, and our analysis focuses on these structural limitations independently of the learning architecture or training procedure.

For a fixed operating condition $\xi$, let
$z\in\mathcal Z\subseteq\mathbb R^p$ denote an admissible LF parameter correction and consider $x_{\mathrm{lf}}(z,\xi)\in$
\begin{equation}\label{eq:lf_z_solution}\argmin_{x_{\mathrm{lf}}}\left\{f\bigl(x_{\mathrm{lf}};\rho_{\mathrm{lf}}+Az,\xi+Bz\bigr)\mid g\bigl(x_{\mathrm{lf}};\rho_{\mathrm{lf}}+Az,\xi+Bz\bigr)\leq0,\;h\bigl(x_{\mathrm{lf}};\rho_{\mathrm{lf}}+Az,\xi+Bz\bigr)=0\right\}. \end{equation}

The LF solution mapping need not be single-valued. We therefore assume that $x_{\mathrm{lf}}(z,\xi)$ denotes a locally unique solution over the parameter values of interest. Uniqueness holds, for example, under strong convexity, and can be enforced in linear or merely convex models through a small strongly convex regularization term \citep{Wilder2019}. More generally, strong regularity of the KKT system yields local uniqueness and Lipschitz continuity of the primal--dual solution mapping; LICQ together with a strong second-order sufficient condition provides standard sufficient conditions \citep{DontchevRockafellar2014}. Whenever derivatives are invoked, we additionally assume the regularity conditions required for local differentiability of the solution map.

Define the corresponding shared LF decision by
$
s_{\mathrm{lf}}(z,\xi)
:=
S_{\mathrm{lf}}x_{\mathrm{lf}}(z,\xi).
\label{eq:lf_shared_solution_map}
$
The set
$
\mathcal R_{\mathrm{lf}}(\xi)
:=
\left\{
s_{\mathrm{lf}}(z,\xi):
z\in\mathcal Z
\right\}
$
contains the shared decisions that can be induced as optimal solutions of the parameterized LF model. We refer to $\mathcal R_{\mathrm{lf}}(\xi)$ as the LF reachable decision set. This set can be substantially smaller than the set of feasible LF decisions. Indeed, the proposed framework does not directly learn a decision; it learns parameters of an optimization problem, and the resulting decision must therefore arise as an optimizer of the corresponding parameterized LF formulation. Consequently, $\mathcal R_{\mathrm{lf}}(\xi)$ captures an intrinsic limitation of the selected LF model family and parameterization.

Let $x_{\mathrm{hf}}^\star(\xi)$ denote a solution of the nominal HF problem and define its shared decision vector as
$
s_{\mathrm{hf}}^\star(\xi)
:=
S_{\mathrm{hf}}x_{\mathrm{hf}}^\star(\xi)
$.
For this HF reference solution, define the best decision-approximation error achievable by the parameterized LF model as
\begin{equation}
\varepsilon_{\mathrm{dec}}(\xi)
:=
\inf_{z\in\mathcal Z}
\left\|
s_{\mathrm{lf}}(z,\xi)
-
s_{\mathrm{hf}}^\star(\xi)
\right\|_2.
\label{eq:decision_approx_gap}
\end{equation}

The quantity $\varepsilon_{\mathrm{dec}}(\xi)$ characterizes the structural decision-approximation capacity of the chosen parameterized LF model. In particular, it provides an architecture-independent lower bound on the decision error achievable by any learned correction map. If
\(
\varepsilon_{\mathrm{dec}}(\xi)>0,
\)
then no correction map $\phi_\theta$, regardless of its expressive power or training procedure, can make the LF optimizer exactly reproduce the selected HF reference decision for that operating condition. The converse need not hold at the level of the learned policy. Even if $\varepsilon_{\mathrm{dec}}(\xi)=0$ for every operating condition individually, a particular function class used to represent $\phi_\theta$ may be unable to realize the corresponding context-dependent corrections simultaneously across all $\xi$. Thus, $\varepsilon_{\mathrm{dec}}(\xi)$ isolates limitations of the parameterized LF formulation from limitations of the learning architecture.

We next characterize locally how the selected LF parameters determine which shared decisions can be modified. Consider a nominal parameter correction
$
z_0\in\operatorname{int}(\mathcal Z),
$
typically $z_0=0$, and suppose that the shared LF solution map is differentiable at $z_0$. Define
$
J_z(\xi)
:=
\nabla_z s_{\mathrm{lf}}(z_0,\xi)
\in\mathbb R^{q\times p}.
$
Then
\begin{equation}
s_{\mathrm{lf}}(z_0+\Delta z,\xi)
=
s_{\mathrm{lf}}(z_0,\xi)
+
J_z(\xi)\Delta z
+
o\left(\lVert\Delta z\rVert_2\right).
\label{eq:lf_local_expansion}
\end{equation}

The matrix $J_z(\xi)$ describes how changes in the LF parameters modify the optimal shared decision.

\begin{proposition}[Locally attainable decision directions]
\label{prop:local_approx_capacity}
Suppose that $z_0\in\operatorname{int}(\mathcal Z)$ and that
$s_{\mathrm{lf}}(\cdot,\xi)$ is differentiable at $z_0$.
A decision-space direction $v\in\mathbb R^q$ can be generated to first order by an admissible infinitesimal perturbation of the LF parameters if and only if
\begin{equation}
v\in\operatorname{Range}\bigl(J_z(\xi)\bigr).
\label{eq:first_order_correctability}
\end{equation}
\end{proposition}

\begin{proof}
For any parameter-space direction $u\in\mathbb R^p$, the interiority of $z_0$ implies that $z_0+tu\in\mathcal Z$ for all sufficiently small $|t|$. Differentiability gives
$$
s_{\mathrm{lf}}(z_0+tu,\xi)
=
s_{\mathrm{lf}}(z_0,\xi)
+
tJ_z(\xi)u
+
o(t).
$$

Hence, every first-order decision variation generated by an admissible infinitesimal perturbation has the form $J_z(\xi)u$ and therefore belongs to $\operatorname{Range}(J_z(\xi))$. Conversely, every vector $v\in\operatorname{Range}(J_z(\xi))$ can be written as $v=J_z(\xi)u$ for some $u$, and is therefore generated to first order by the perturbation $tu$.
\end{proof}

This result provides a local characterization of the expressive capacity of the selected LF parameterization. In particular, consider the discrepancy between the nominal LF and HF decisions,
\begin{equation}
d(\xi)
:=
s_{\mathrm{hf}}^\star(\xi)
-
s_{\mathrm{lf}}(z_0,\xi).
\label{eq:nominal_lf_hf_discrepancy}
\end{equation}
Define the linearized approximation residual
\begin{equation}
r_{\mathrm{lin}}(\xi) :=
\min_{\Delta z}
\lVert
d(\xi)-J_z(\xi)\Delta z
\rVert_2 =
\lVert
\left(
I-J_z(\xi)J_z(\xi)^\dagger
\right)d(\xi)
\rVert_2,
\label{eq:linearized_approx_residual}
\end{equation}
where $J_z(\xi)^\dagger$ denotes the Moore--Penrose pseudoinverse.

The quantity $r_{\mathrm{lin}}(\xi)$ is a local diagnostic of how well the nominal LF--HF discrepancy aligns with the first-order decision directions generated by the selected LF parameters. It should not be interpreted as the exact approximation error $\varepsilon_{\mathrm{dec}}(\xi)$, since the parameter perturbation required by the linearized model may be too large for the first-order approximation to remain accurate, and nonlinear changes in the LF solution map may generate additional decision variations away from $z_0$. The matrices $A$ and $B$ affect this local approximation capacity through $J_z(\xi)$. In particular, if the shared LF solution map is viewed as a function of the LF parameters and operating inputs, then, whenever the corresponding derivatives exist,
\begin{equation}
J_z(\xi)
=
D_{\rho}s_{\mathrm{lf}}\,A
+
D_{\xi}s_{\mathrm{lf}}\,B,
\label{eq:lf_parameter_jacobian_decomposition}
\end{equation}
where the derivatives are evaluated at $\bigl(\rho_{\mathrm{lf}}+Az_0,\xi+Bz_0\bigr)$. Thus, selecting which LF quantities are tunable determines not only the dimension of the learned correction, but also which directions of the optimal LF decision can be modified locally. The connection between decision approximation and HF evaluation can be established without requiring the projected feasible regions of the LF and HF models to coincide. Let $\mathcal V(s,\xi)$ denote the optimal value of the slack-augmented HF evaluation problem when the shared LF decision is fixed at $s$.

The Lipschitz assumption below is a standard stability property of parametric nonlinear programs. For continuously differentiable problem data, local Lipschitz continuity of the optimal value function follows under standard constraint qualifications and stability conditions, including local uniform compactness of the feasible set and MFCQ at the relevant optimal solutions \citep{GauvinDubeau1982,BonnansShapiro2000}. Stronger regularity of the KKT system also yields locally Lipschitz primal--dual solution mappings \citep{DontchevRockafellar2014}. Because Since Proposition~\ref{prop:decision_to_operational} requires Lipschitz continuity over the relevant LF-reachable set, we impose it directly instead of relying on stronger NLP regularity conditions.

\begin{proposition}[Decision and evaluation approximation]
\label{prop:decision_to_operational}
Suppose that, for a fixed $\xi$, $\mathcal V(\cdot,\xi)$ is
$L_{\mathcal V}(\xi)$-Lipschitz over a set containing
$s_{\mathrm{hf}}^\star(\xi)$ and $\mathcal R_{\mathrm{lf}}(\xi)$. 
Then
\begin{equation}
\inf_{z\in\mathcal Z}
\left|
\mathcal V\bigl(s_{\mathrm{lf}}(z,\xi),\xi\bigr)
-
\mathcal V\bigl(s_{\mathrm{hf}}^\star(\xi),\xi\bigr)
\right|
\leq
L_{\mathcal V}(\xi)
\varepsilon_{\mathrm{dec}}(\xi).
\label{eq:evaluation_approx_bound}
\end{equation}
\end{proposition}

\begin{proof}
For every $z\in\mathcal Z$, Lipschitz continuity gives
$$
\left|
\mathcal V\bigl(s_{\mathrm{lf}}(z,\xi),\xi\bigr)
-
\mathcal V\bigl(s_{\mathrm{hf}}^\star(\xi),\xi\bigr)
\right|
\leq
L_{\mathcal V}(\xi)
\left\|
s_{\mathrm{lf}}(z,\xi)
-
s_{\mathrm{hf}}^\star(\xi)
\right\|_2.
$$

Taking the infimum over $z\in\mathcal Z$ yields
\eqref{eq:evaluation_approx_bound}.
\end{proof}

Proposition~\ref{prop:decision_to_operational} shows that a small decision approximation error implies a small HF evaluation gap under the Lipschitz condition. The reverse is not guaranteed, since a different reachable decision may achieve similar or better HF performance. Thus, $\varepsilon_{\mathrm{dec}}(\xi)$ measures how well the LF model can reproduce the HF reference decision, not the best HF performance attainable over $\mathcal R_{\mathrm{lf}}(\xi)$. Moreover, $\mathcal V(s_{\mathrm{hf}}^\star(\xi),\xi)$ need not equal the nominal HF optimum unless additional conditions hold on the coupling constraints and slack penalties.

Together, $\mathcal R_{\mathrm{lf}}(\xi)$, $\varepsilon_{\mathrm{dec}}(\xi)$, $J_z(\xi)$, and $r_{\mathrm{lin}}(\xi)$ characterize the structural approximation capacity of the parameterized LF model: its reachable decisions, global approximation error, locally attainable directions, and alignment of the LF--HF discrepancy with those directions. These limitations are independent of the learning architecture: a more expressive correction map cannot overcome restrictions of the LF reachable set, while an insufficiently expressive map may fail to exploit its available capacity across operating conditions.

\subsection{Supervised Learning}\label{sec:supervised}
A supervised learning-based approach was proposed by \cite{Constante-Flores2025} for AC-OPF and DC-OPF applications. To make this paper self-contained, we summarize the approach using the general HF-LF notations introduced.
\paragraph{Generation of HF reference decisions.}
Given the $\{\xi_i\}_{i=1}^{n}$, we first solve the nominal HF problem for each sample. The resulting solution $x_{\mathrm{hf},i}^{\star}$ serves as the reference decision for $\xi_i$. 
\paragraph{Sample-wise generation of calibration labels.}
For each sample $i$, let $z_i\in\mathbb{R}^{p}$ denote a realization-specific correction to the LF model parameters. The ideal correction $z_i^\star$ is obtained by solving
\begin{subequations}\label{eq:samplewise_calibration}
\begin{align}
    \min_{z_i,\,x_{\mathrm{lf},i}}
    \quad &
    \left\|
       S_{\mathrm{lf}} x_{\mathrm{lf},i}
        -
        S_{\mathrm{hf}}x_{\mathrm{hf},i}^{\star}
    \right\|_2^2
    \label{eq:samplewise_calibration_a}\\
    \text{s.t.}\quad
    &
    x_{\mathrm{lf},i}
    \in
    \argmin_{x_{\mathrm{lf}}}
    \left\{
    \begin{aligned}
        \quad &
        f\bigl(
            x_{\mathrm{lf}};
            \rho_{\mathrm{lf}}+Az_i,
            \xi_i+Bz_i
        \bigr)
        \\
        \text{s.t.}\quad &
        g\bigl(
            x_{\mathrm{lf}};
            \rho_{\mathrm{lf}}+Az_i,
            \xi_i+Bz_i
        \bigr)
        \leq 0,
        \\
        &
        h\bigl(
            x_{\mathrm{lf}};
            \rho_{\mathrm{lf}}+Az_i,
            \xi_i+Bz_i
        \bigr)
        = 0
    \end{aligned}
    \right\}.
    \label{eq:samplewise_calibration_b}
\end{align}
\end{subequations}

We denote an optimal solution of \eqref{eq:samplewise_calibration} by $\bigl(z_i^\star,x_{\mathrm{lf},i}^\star\bigr)$. The objective in \eqref{eq:samplewise_calibration_a} encourages the tuned LF decision to reproduce the corresponding HF reference decision. Problem \eqref{eq:samplewise_calibration} is a sample-wise bilevel problem. Unlike \eqref{eq:empirical_bilevel}, however, the calibration problems are independent across samples and can therefore be solved in parallel. The solutions
$
    \mathcal{D}_{\mathrm{sup}}
    =
    \left\{
        \bigl(\xi_i,z_i^\star\bigr)
    \right\}_{i=1}^{n}
$ form a supervised-learning dataset in which the $\xi_i$ is the input and the ideal LF correction $z_i^\star$ is the target label.

\paragraph{Supervised training of the correction map.}
We train $\phi_\theta$ to approximate the mapping from a $\xi_i$ to its sample-specific correction $z_i^\star$. The supervised-learning problem is
\begin{equation}
    \min_{\theta\in\Theta}
    \quad
    \widehat{\mathcal{L}}_{\mathrm{sup}}(\theta)
    :=
    \frac{1}{n}
    \sum_{i=1}^{n}
    \left\|
        \phi_\theta(\xi_i)-z_i^\star
    \right\|_2^2
    \label{eq:supervised_training}
\end{equation}

This approach replaces the joint end-to-end optimization of
\eqref{eq:empirical_bilevel} with an offline label-generation procedure
followed by standard supervised learning. 

\subsection{Semi-Supervised Learning}
\label{sec:semi_supervised}

The supervised-learning approach requires solving the sample-wise calibration problem \eqref{eq:samplewise_calibration} to generate a correction label $z_i^\star$ for every training sample. Semi-supervised learning avoids this label-generation step and instead requires only the HF reference decisions $\{x_{\mathrm{hf},i}^{\star}\}_{i=1}^{n}$, obtained by solving the nominal HF problem for each training context. The correction map $\phi_\theta$ is then trained end-to-end by differentiating through the tuned LF optimization problem.

For each sample $\xi_i$, the network produces the correction $z_i=\phi_\theta(\xi_i)$, which determines the tuned LF parameters through \eqref{eq:tuned_parameters}. Solving the resulting tuned LF problem \eqref{eq:lf_nominal_tuned} gives the decision $x_{\mathrm{lf},i}^{\theta}$. The semi-supervised loss is
\begin{equation}
    \min_{\theta\in\Theta}
    \quad
    \widehat{\mathcal{L}}_{\mathrm{semi}}(\theta)
    :=
    \frac{1}{n}
    \sum_{i=1}^{n}
    \left\|
        S_{\mathrm{lf}}x_{\mathrm{lf},i}^{\theta}
        -
        S_{\mathrm{hf}}x_{\mathrm{hf},i}^{\star}
    \right\|_2^2.
    \label{eq:semi_objective}
\end{equation}
Thus, the network is trained to make the common components of the tuned LF
decision reproduce those of the corresponding HF-optimal decision, without
first constructing correction labels.

The computational graph for each sample is
\begin{equation}
    \xi_i
    \longrightarrow
    \phi_\theta(\xi_i)
    \longrightarrow
    \bigl(
        \rho_{\theta}^{\mathrm{lf}}(\xi_i),
        \xi_{\theta}(\xi_i)
    \bigr)
    \longrightarrow
    x_{\mathrm{lf},i}^{\theta}
    \longrightarrow
    \ell_{\mathrm{semi},i},
    \label{eq:semi_computational_graph}
\end{equation}
where 
\(
    \ell_{\mathrm{semi},i}
    :=
    \left\|
        S_{\mathrm{lf}}x_{\mathrm{lf},i}^{\theta}
        -
        S_{\mathrm{hf}}x_{\mathrm{hf},i}^{\star}
    \right\|_2^2.
\)

Applying the chain rule gives
\begin{equation}
    \nabla_\theta \ell_{\mathrm{semi},i}=
    \left(
        \frac{\partial x_{\mathrm{lf},i}^{\theta}}
             {\partial\theta}
    \right)^{\!\top}
    \nabla_{x_{\mathrm{lf}}}
    \ell_{\mathrm{semi},i},
    \quad
    \frac{\partial x_{\mathrm{lf},i}^{\theta}}
         {\partial\theta} =
    \frac{\partial x_{\mathrm{lf},i}^{\theta}}
         {\partial z_i}
    \frac{\partial\phi_\theta(\xi_i)}
         {\partial\theta}.
\label{eq:semi_chain_rule}
\end{equation}
The Jacobian
$\partial x_{\mathrm{lf},i}^{\theta}/\partial z_i$ is obtained by implicitly differentiating the optimality conditions of the tuned LF problem. In practice, a differentiable optimization layer computes the required vector--Jacobian product by solving an adjoint linear system associated with the LF optimality conditions \citep{Besancon2024}. The gradient is then propagated through $\phi_\theta$ using automatic differentiation.

\subsection{Self-Supervised Learning}
\label{sec:self_supervised}

The supervised and semi-supervised approaches require offline HF solutions, either to generate correction labels or to provide decision targets. Self-supervised learning eliminates this requirement by training the correction map directly using the HF evaluation loss $\ell_{\mathrm{hf}}$ defined by the slack-augmented HF problem \eqref{eq:hf_problem_eval}.

For each training sample $\xi_i$, the correction $z_i=\phi_\theta(\xi_i)$ determines the tuned LF parameters through \eqref{eq:tuned_parameters}, and solving \eqref{eq:lf_nominal_tuned} gives the decision $x_{\mathrm{lf}}^\theta(\xi_i)$. The self-supervised training problem is
\begin{equation}
    \min_{\theta\in\Theta}
    \quad
    \widehat{\mathcal{L}}_{\mathrm{self}}(\theta)
    :=
    \frac{1}{n}
    \sum_{i=1}^{n}
    \ell_{\mathrm{hf}}
    \bigl(
        x_{\mathrm{lf}}^\theta(\xi_i),
        \xi_i
    \bigr).
    \label{eq:self_training}
\end{equation}
Because $\ell_{\mathrm{hf}}$ is the optimal value of \eqref{eq:hf_problem_eval}, the training loss accounts for both the HF operational objective and the feasibility corrections required to reconcile the LF decision with the HF model.

Applying the chain rule to each term in \eqref{eq:self_training} gives
\begin{equation}
    \nabla_\theta
    \ell_{\mathrm{hf}}
    \bigl(
        x_{\mathrm{lf}}^\theta(\xi_i),
        \xi_i
    \bigr)
    =
    \left(
        \frac{\partial x_{\mathrm{lf}}^\theta(\xi_i)}
             {\partial\theta}
    \right)^{\!\top}
    \nabla_{x_{\mathrm{lf}}}
    \ell_{\mathrm{hf}}
    \bigl(
        x_{\mathrm{lf}}^\theta(\xi_i),
        \xi_i
    \bigr),
    \qquad
    \frac{\partial x_{\mathrm{lf}}^\theta(\xi_i)}
         {\partial\theta}
    =
    \frac{\partial x_{\mathrm{lf}}^\theta(\xi_i)}
         {\partial z_i}
    \frac{\partial\phi_\theta(\xi_i)}
         {\partial\theta}.
    \label{eq:self_chain_rule}
\end{equation}
The Jacobian
$\partial x_{\mathrm{lf}}^\theta(\xi_i)/\partial z_i$
is obtained by implicitly differentiating the optimality conditions of the tuned LF problem \citep{Besancon2024}. The gradient of the HF evaluation value with respect to the LF decision can be obtained from the optimal multipliers associated with the constraints through which the LF decision enters the HF evaluation problem, under standard regularity conditions for parametric optimization \citep{BonnansShapiro2000}. These sensitivities are propagated through the LF solution map and $\phi_\theta$, without differentiating through the iterations of the HF solver.

Self-supervised learning therefore directly minimizes the HF operational loss induced by the tuned LF model and is the closest of the three approaches to the empirical objective \eqref{eq:empirical_bilevel}. Its principal computational disadvantage is that both the LF and HF optimization problems must be solved during training.
\subsection{Interpretation of the Learning Objectives}
\label{subsec:operational_interpretation}

The semi-supervised and self-supervised formulations evaluate the quality of an LF decision using different criteria. For an operating condition $\xi$, let
\(
s^\theta(\xi)
:=
S_{\mathrm{lf}}x_{\mathrm{lf}}^\theta(\xi)
\)
denote the shared components of the decision produced by the tuned LF model, and let
\(
s_{\mathrm{hf}}^\star(\xi)
:=
S_{\mathrm{hf}}x_{\mathrm{hf}}^\star(\xi)
\)
denote the corresponding shared components of an HF-optimal decision. The semi-supervised formulation minimizes the Euclidean decision error
\begin{equation}
\left\|
s^\theta(\xi)-s_{\mathrm{hf}}^\star(\xi)
\right\|_2^2,
\label{eq:interp_semi_decision_loss}
\end{equation}
whereas the self-supervised learning evaluates the LF decision through the HF evaluation value,
\begin{equation}
\mathcal V\bigl(s^\theta(\xi),\xi\bigr).
\label{eq:interp_self_hf_value}
\end{equation}
Thus, semi-supervised learning seeks agreement with a particular HF decision, whereas self-supervised learning seeks good operational performance under the HF model.

This distinction becomes important when the parameterized LF model cannot reproduce the HF-optimal decision exactly. Recall from Section~\ref{subsec:lf_approximation_capacity} that
\(
\mathcal R_{\mathrm{lf}}(\xi)
=
\left\{
s_{\mathrm{lf}}(z,\xi):
z\in\mathcal Z
\right\}
\)
contains the shared decisions that can be induced as optimal solutions of the parameterized LF model. If
\(
s_{\mathrm{hf}}^\star(\xi)
\notin
\mathcal R_{\mathrm{lf}}(\xi),
\)
then no admissible LF parameter correction can exactly reproduce the HF decision.

Ignoring restrictions imposed by the particular function class used to represent $\phi_\theta$, the idealized semi-supervised problem for a fixed operating condition is
\begin{equation}
\min_{s\in\mathcal R_{\mathrm{lf}}(\xi)}
\left\|
s-s_{\mathrm{hf}}^\star(\xi)
\right\|_2^2.
\label{eq:interp_reachable_semi}
\end{equation}
In contrast, the corresponding idealized self-supervised problem is
\begin{equation}
\min_{s\in\mathcal R_{\mathrm{lf}}(\xi)}
\mathcal V(s,\xi).
\label{eq:interp_reachable_self}
\end{equation}
Therefore, the two approaches generally seek different decisions within the same LF-reachable set. Semi-supervised learning selects an attainable decision that is closest to the HF reference decision under the chosen decision-space metric, whereas self-supervised learning selects an attainable decision that minimizes the HF evaluation loss. Consequently, when the HF reference decision is not attainable, its closest LF approximation need not be the attainable decision with the best HF operational performance. This distinction also admits a local geometric interpretation. Consider a nominal parameter correction $z_0$ and define
\[
s_0(\xi)
:=
s_{\mathrm{lf}}(z_0,\xi),
\qquad
d(\xi)
:=
s_{\mathrm{hf}}^\star(\xi)-s_0(\xi).
\]
As established in Section~\ref{subsec:lf_approximation_capacity}, for a sufficiently small parameter perturbation $\Delta z$,
\begin{equation}
s_{\mathrm{lf}}(z_0+\Delta z,\xi)
=
s_0(\xi)
+
J_z(\xi)\Delta z
+
o\left(\lVert \Delta z\rVert_2\right).
\label{eq:interp_lf_local_expansion}
\end{equation}
Substituting this approximation into the semi-supervised loss yields, to first order,
\begin{equation}
\min_{\Delta z}
\left\|
J_z(\xi)\Delta z-d(\xi)
\right\|_2^2.
\label{eq:interp_local_semi}
\end{equation}
Hence, locally, semi-supervised learning seeks to reduce the component of the LF--HF decision discrepancy that can be modified through the selected LF parameters, with discrepancy measured in Euclidean distance.

Under additional regularity conditions, the self-supervised criterion admits a related local interpretation. Suppose, as a special case, that $\mathcal V(\cdot,\xi)$ is twice continuously differentiable in a neighborhood of $s_{\mathrm{hf}}^\star(\xi)$ and that $s_{\mathrm{hf}}^\star(\xi)$ is a local minimizer of $\mathcal V(\cdot,\xi)$. Then
\begin{equation}
\nabla_s
\mathcal V\bigl(s_{\mathrm{hf}}^\star(\xi),\xi\bigr)
=
0,
\label{eq:interp_hf_stationarity}
\end{equation}
and, for $s$ sufficiently close to $s_{\mathrm{hf}}^\star(\xi)$,
\begin{equation}
\mathcal V(s,\xi)
-
\mathcal V\bigl(s_{\mathrm{hf}}^\star(\xi),\xi\bigr)
=
\frac{1}{2}
\bigl(s-s_{\mathrm{hf}}^\star(\xi)\bigr)^\top
Q(\xi)
\bigl(s-s_{\mathrm{hf}}^\star(\xi)\bigr)
+
o\!\left(
\left\|
s-s_{\mathrm{hf}}^\star(\xi)
\right\|_2^2
\right),
\label{eq:interp_hf_second_order}
\end{equation}
where 
\(
Q(\xi)
:=
\nabla_{ss}^2
\mathcal V\bigl(s_{\mathrm{hf}}^\star(\xi),\xi\bigr)
\succeq 0.
\)
In the local regime where the LF decisions under consideration remain close to $s_{\mathrm{hf}}^\star(\xi)$, combining \eqref{eq:interp_lf_local_expansion} and \eqref{eq:interp_hf_second_order} gives the leading-order approximation
\begin{equation}
\min_{\Delta z}
\frac{1}{2}
\left(
J_z(\xi)\Delta z-d(\xi)
\right)^\top
Q(\xi)
\left(
J_z(\xi)\Delta z-d(\xi)
\right).
\label{eq:interp_local_self}
\end{equation}

Equations~\eqref{eq:interp_local_semi} and \eqref{eq:interp_local_self} provide a local geometric interpretation of the two objectives. The sensitivity matrix $J_z(\xi)$ determines the first-order attainable decision directions generated by the selected LF parameters, while the learning objective determines how discrepancies along those directions are evaluated. Semi-supervised learning uses Euclidean distance, whereas self-supervised learning uses a curvature-weighted quadratic form determined by $Q(\xi)$. The two criteria are locally aligned when $Q(\xi)$ acts approximately as a scalar multiple of the identity over the relevant decision directions; otherwise, they generally prioritize different components of the LF--HF discrepancy.

The curvature interpretation is not required for the broader distinction between the two objectives. Equations~\eqref{eq:interp_reachable_semi}--\eqref{eq:interp_reachable_self} already show that semi-supervised learning minimizes decision discrepancy over the LF-reachable set, whereas self-supervised learning directly minimizes HF evaluation over the same set. This is consistent with the decision-focused learning perspective that matching decisions does not necessarily optimize downstream performance \citep{Elmachtoub2022,HoNguyenKilincKarzan2022}. When the LF model cannot reproduce the HF-optimal decision exactly, self-supervised learning can therefore use the available LF flexibility to improve HF performance directly, although training may be more difficult because the HF evaluation value can be nonconvex, nonsmooth, locally flat, or sensitive to local solution branches. This helps explain the numerical results in Section~\ref{sec:Results}: when the HF reference decision is not LF-reachable, the closest reachable decision in Euclidean distance need not yield the best HF evaluation value.

\subsection{Special Considerations for Network Flow Problems}
\label{sec:network_flow_considerations}

For network flow problems, the correction map should account explicitly for the underlying graph structure. Rather than applying a standard multilayer perceptron to a flattened representation of the operating conditions, $\phi_\theta$ can be parameterized using a graph neural network (GNN) \citep{cappart2023combinatorial}. Nodes may represent network components such as buses, sources, or sinks, while edges represent physical connections and their associated capacities, costs, or other attributes.

A GNN constructs the correction through repeated local message-passing operations. Consequently, its  parameterization is naturally compatible with the locality and permutation symmetries of network flow models. The same learned parameters can also be applied across networks with different numbers of nodes and edges, provided that the node and edge features have consistent meanings. This structural inductive bias can improve sample efficiency and generalization to operating conditions or network topologies that are not observed during training. In the proposed framework, the GNN output can be used to generate node-wise, edge-wise, or global corrections to the LF model parameters. The remainder of the supervised, semi-, or self-supervised learning procedure remains unchanged; only the architecture used to represent $\phi_\theta$ is modified.

\section{Results and Discussion}\label{sec:Results}

This section presents the experimental results of the proposed self-supervised and semi-supervised frameworks for three optimization problems: optimal power flow (OPF), gas network optimization, and the pooling problem. All the CPU experiments are conducted on a single thread of a Linux server equipped with an AMD EPYC 7643 processor (48 cores, 96 threads and 1TB of RAM. Additionally, the GPU-based OPF framework is evaluated on an AMD Ryzen AI 9 HX 370 processor with 32~GB of RAM, and an NVIDIA GeForce RTX 5070 Ti GPU.

The experiments are implemented in Julia~v1.10.11. The GNNs are developed with \texttt{GraphNeuralNetworks.jl} (v0.6.23) and trained using \texttt{Flux.jl} (v0.14.21), while differentiable optimization layers used \texttt{DiffOpt.jl} (v0.5.6). GPU-accelerated OPF experiments are implemented with \texttt{ExaModels.jl} (v0.11.2). We implement the semi and self-supervised learning frameworks for all three applications. For both learning frameworks, a GNN is trained offline for each optimization problem and fidelity setting using an application-specific number of training samples: 1 sample for both the optimal power flow and optimal gas flow problems, and 25 samples for the pooling problem, then evaluated online on 1,000 unseen samples drawn from the training distribution. The code can be found at \url{https://github.com/li-group/HiFi_LoFi_learning}.

\paragraph{Evaluation metrics.}
We evaluate all LF models according to their performance when their resulting decisions are passed to the corresponding HF formulation. Two primary metrics are used across the three applications, and reported on unseen test instances.

First, \emph{objective error} measures the mismatch between the HF evaluation loss induced by the LF decision and the optimal objective value of the nominal HF model. For a test instance $\xi$, we define
\begin{equation}
e_{\mathrm{obj}}(\xi)
:=
100\cdot
\frac{
\left|
\ell_{\mathrm{hf}}
\bigl(x_{\mathrm{lf}}^\theta(\xi),\xi\bigr)
-
F_{\mathrm{hf}}^{\mathrm{ref}}(\xi)
\right|
}{
\left|F_{\mathrm{hf}}^{\mathrm{ref}}(\xi)\right|
}.
\label{eq:evaluation_objective_error}
\end{equation}
where $F_{\mathrm{hf}}^{\mathrm{ref}}(\xi)$ is the objective value of the reference solution obtained by solving the nominal HF problem. Since the HF models considered here may be nonconvex, this value should be interpreted as a high-quality reference value, e.g., the optimal objective value of a local solver, rather than a certified global optimum.

Second, \emph{feasibility violation} quantifies the magnitude of these corrections. Because slack variables have different meanings and units across applications, for each sample we compute the application-specific mean constraint violation from the corresponding HF evaluation problem. We also report the 90th percentile (P90) of this distribution across samples, capturing the magnitude of corrections for most samples while reducing sensitivity to extreme cases.

\subsection{Optimal Power Flow}
The OPF determines the optimal steady-state operating point of a power system by minimizing generation cost subject to the power flow equations and operational limits on generators, voltages, and network elements. In our setup, the AC-OPF serves as the HF model and is formulated as an NLP problem, whereas the LF model corresponds to the DC-OPF, formulated as a quadratic programming (QP) problem based on a linear approximation of the power flow equations that neglects active power losses. We use the active power demands in the DC-OPF as the tunable parameters, allowing the learned parameters to implicitly account for the losses neglected by the DC approximation. We consider power system test cases ranging from 500 to 30,000 buses, whose problem size in terms of variables and constraints are detailed in Table~\ref{tab:opf_problem_size} of Appendix~\ref{app:opf}.

We compare the original (untuned) DC-OPF, the parameterized DC-OPF (proposed), and the AC-OPF, which serves as the HF reference model. Test instances are generated by perturbing the nominal active power demand by $\pm30\%$, while the reactive power demand is varied between $0\%$ and $-15\%$. We evaluate the proposed GNN under both transductive and inductive learning settings.

\subsubsection{Results}
We use a GNN with node and branch features representing the power system operating condition; architecture and training details are provided in Table~{\ref{tab:gnn_architecture_opf}} Appendix~\ref{app:opf}. At test time, the GNN predicts the LF model parameters, the resulting DC-OPF is solved, and its active power dispatch is evaluated in the AC-OPF model with slack variables representing corrective actions needed to recover feasibility. Table~\ref{tab:opf_transductive} reports the relative objective error and P90 feasibility violation. The self-supervised approach performs best across all systems, achieving a mean relative objective error below $0.4\%$ and P90 feasibility violation below $0.5$~MW, while the semi-supervised approach yields less consistent improvements, particularly for larger systems.

\begin{table}[H]
\centering
\TABLE{HF evaluation of the transductive OPF experiments. Objective errors are reported as mean (maximum) relative error.
\label{tab:opf_transductive}}
{\begin{tabular}{l|cc|cc|cc}
\hline \up
\multirow{2}{*}{\textbf{Case}} &
\multicolumn{2}{c|}{\textbf{Self-Supervised}} &
\multicolumn{2}{c|}{\textbf{Semi-Supervised}} &
\multicolumn{2}{c}{\textbf{Untuned}} \\
\cline{2-7} \up
&
\begin{tabular}[c]{@{}c@{}}Error (\%)\\Mean (Max)\end{tabular} &
\begin{tabular}[c]{@{}c@{}}P90 Feas.\\(MW)\end{tabular} &
\begin{tabular}[c]{@{}c@{}}Error (\%)\\Mean (Max)\end{tabular} &
\begin{tabular}[c]{@{}c@{}}P90 Feas.\\(MW)\end{tabular} &
\begin{tabular}[c]{@{}c@{}}Error (\%)\\Mean (Max)\end{tabular} &
\begin{tabular}[c]{@{}c@{}}P90 Feas.\\(MW)\end{tabular} \\
\hline \up

1,354
& 0.314 (4.617) & 0.317
& 1.187 (14.351) & 2.430
& 3.311 (7.950) & 6.253 \\

2,000
& 0.153 (0.277) & 0.131
& 0.321 (0.464) & 0.397
& 2.474 (2.579) & 3.422 \\

4,661
& 0.276 (3.781) & 0.176
& 0.393 (3.955) & 0.380
& 1.736 (5.120) & 1.802 \\

13,659
& 0.293 (0.453) & 0.109
& 1.678 (1.826) & 1.789
& 2.087 (2.231) & 2.306 \\

30,000
& 0.320 (0.414) & 0.478
& 0.509 (0.635) & 0.540
& 4.404 (4.482) & 1.260 \down \\

\hline
\end{tabular}
}{}
\end{table}

Figure~\ref{fig:opf_evaluation} summarizes the cost and feasibility results across the OPF test systems. The cost correlations show that the GNN generalizes from a single training operating point to a broad range of unseen conditions. The self-supervised approach also performs more consistently than the semi-supervised approach, indicating that directly incorporating HF feedback during training is more effective than matching HF solution labels alone. The feasibility results show that the self-supervised approach consistently reduces violations and the corrective actions required by the HF model relative to the untuned model. These improvements are particularly pronounced for the larger systems, where the self-supervised approach also outperforms the semi-supervised model.

\begin{figure*}[htbp]
\centering

\begin{overpic}[width=\textwidth]{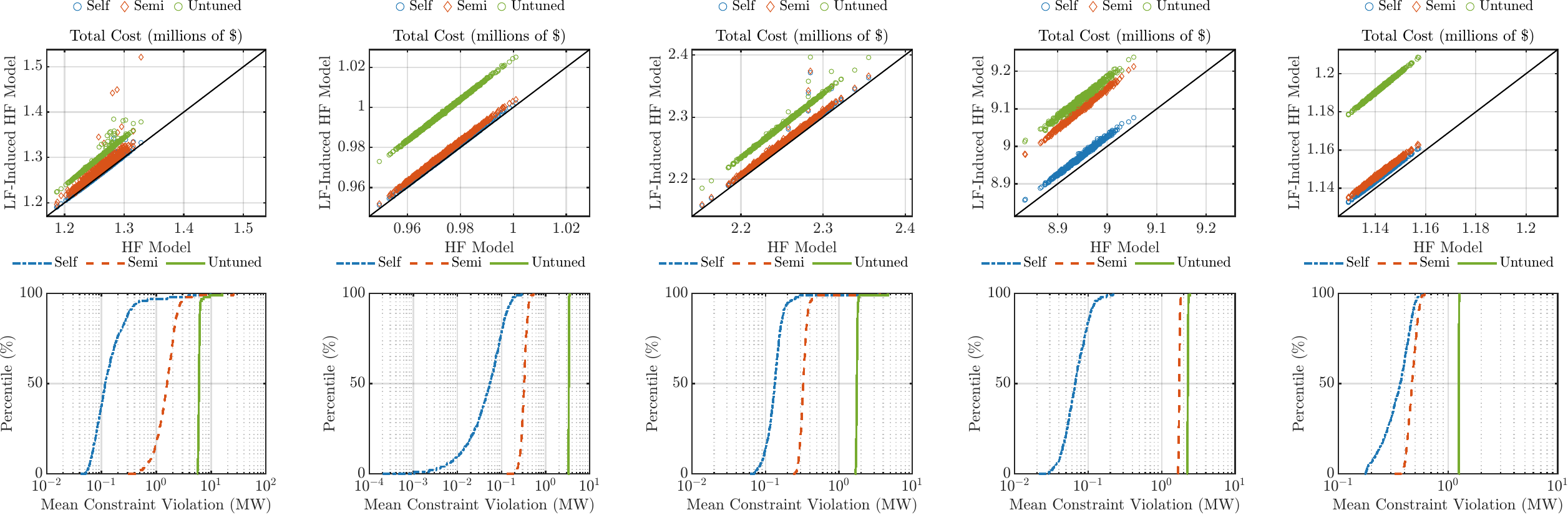}

    \put(10,-2){\makebox(0,0){\footnotesize (a) 1,354-bus system}}
    \put(30,-2){\makebox(0,0){\footnotesize (b) 2,000-bus system}}
    \put(50.5,-2){\makebox(0,0){\footnotesize (c) 4,661-bus system}}
    \put(71,-2){\makebox(0,0){\footnotesize (d) 13,659-bus system}}
    \put(91.25,-2){\makebox(0,0){\footnotesize (e) 30,000-bus system}}

\end{overpic}

\vspace{1em}

\caption{Cost and feasibility evaluation for the OPF benchmark instances. The top panels show cost correlation, and the bottom panels show feasibility across the test samples. 
}
\label{fig:opf_evaluation}
\end{figure*}

Additionally, we evaluate the inductive capability of the GNN by training on the 500-bus system and testing on larger systems, up to 30,000 buses. As shown in Figure~\ref{fig:vio_correlation}, the inductive approach consistently outperforms the untuned model, achieving relative errors below 2\% across all systems.

\begin{figure*}[ht]
\centering
\subfloat[1,354-bus system]{\includegraphics[width=0.19\textwidth,trim=0 0 0 25,
    clip]{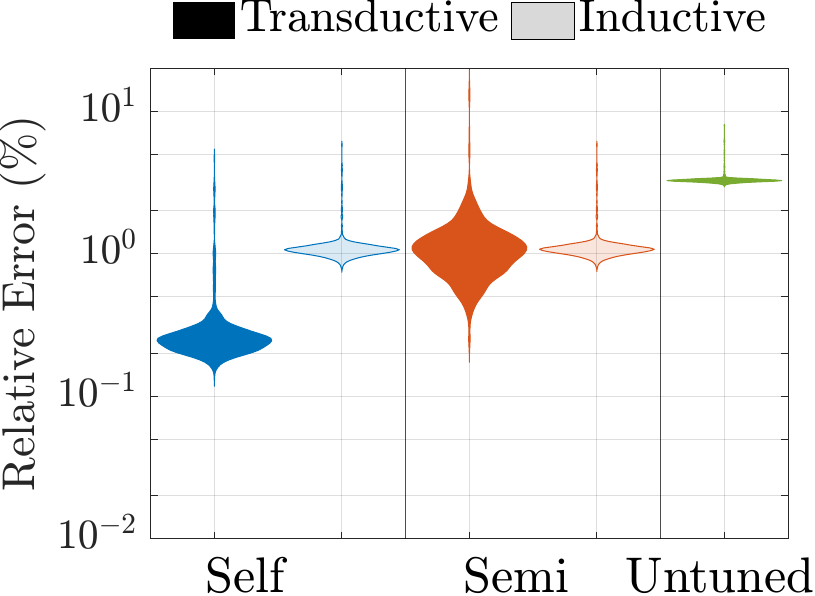}\label{fig:vio_1354}}
\hfill
\subfloat[2,000-bus system]{\includegraphics[width=0.19\textwidth,trim=0 0 0 25,
    clip]{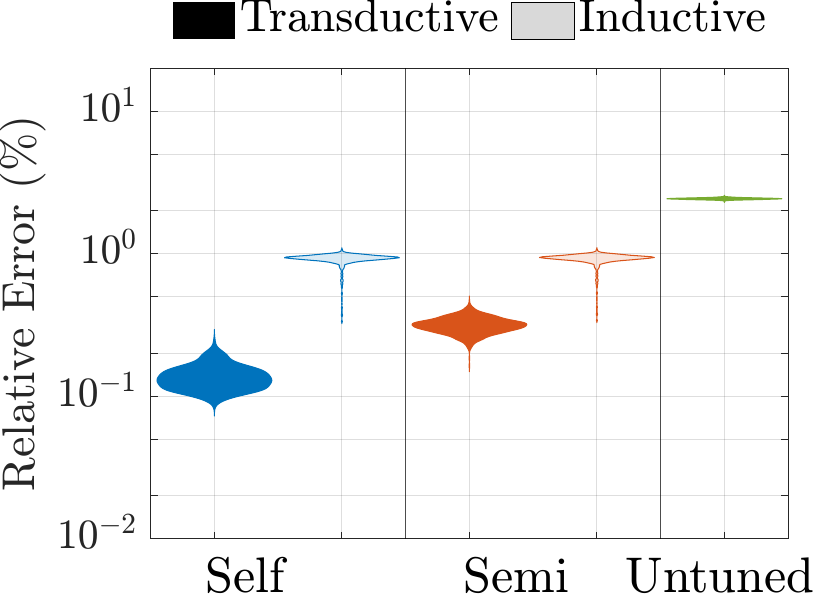}\label{fig:vio_2000}}
\hfill
\subfloat[4,661-bus system]{\includegraphics[width=0.19\textwidth,trim=0 0 0 25,
    clip]{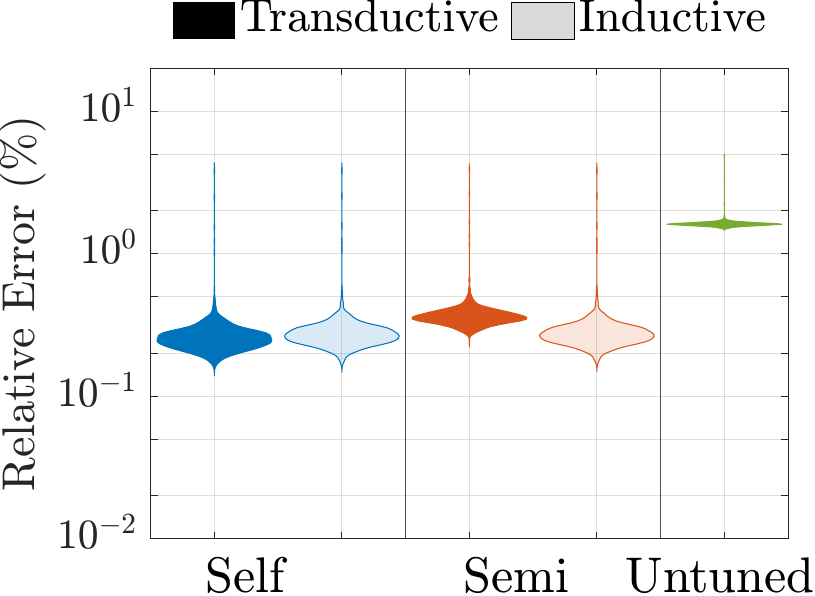}\label{fig:vio_4661}}
\hfill
\subfloat[13,659-bus system]{\includegraphics[width=0.19\textwidth,trim=0 0 0 25,
    clip]{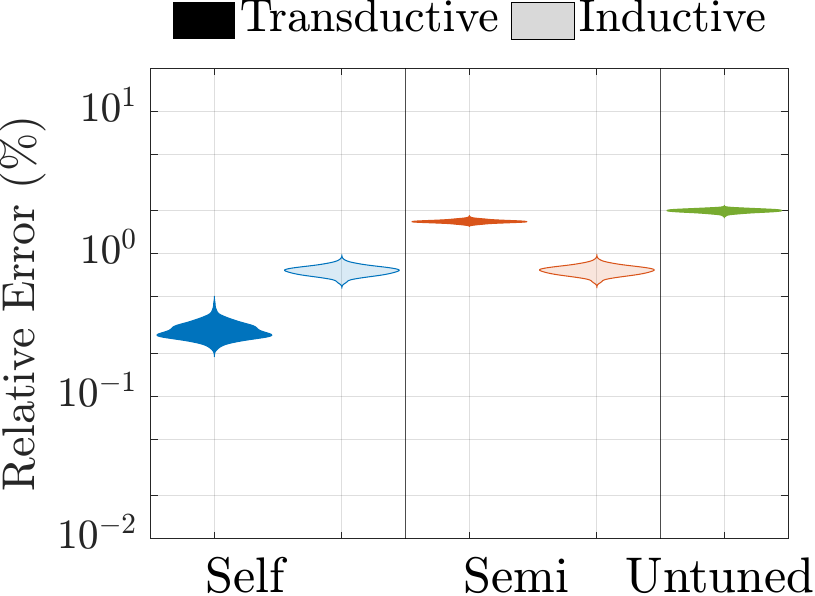}\label{fig:vio_13659}}
\hfill
\subfloat[30,000-bus system]{\includegraphics[width=0.19\textwidth,trim=0 0 0 32,
    clip]{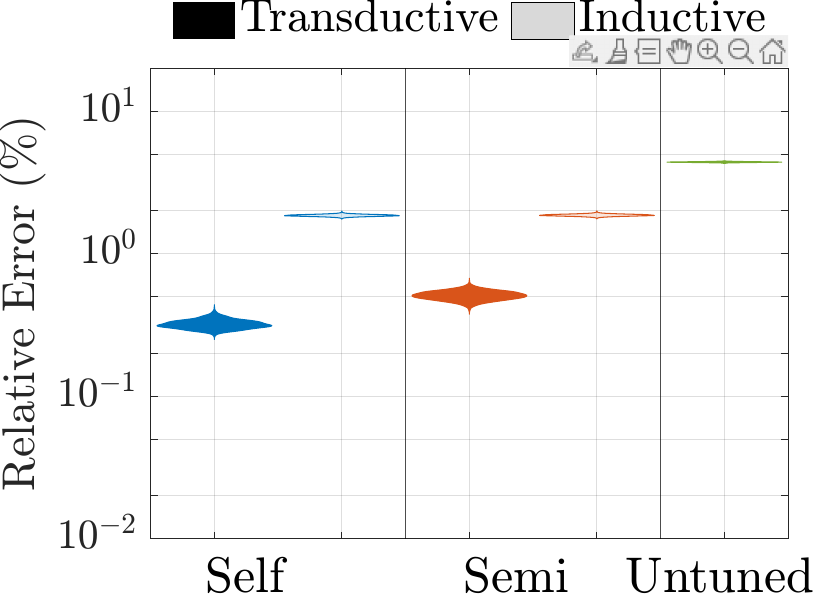}\label{fig:vio_30000}}

\caption{Relative error across systems under transductive and inductive learning. Filled and unfilled violins indicate transductive and inductive results, respectively.}
\label{fig:vio_correlation}
\end{figure*}

Since the low-fidelity model is formulated as a QP and solved on a CPU, its computational efficiency enables significantly faster solution times than the AC-OPF, as shown in Table~\ref{tab:inference_times_transductive}. For the 13,659-bus system, the proposed approach is approximately $26\times$ faster than directly solving the AC-OPF. For the 30,000-bus system, both the proposed approach and the AC-OPF are GPU-implemented. Even under this GPU-accelerated setting, the proposed approach remains approximately $3.5\times$ faster, which highlights its computational advantage at large scale.
\begin{table}[H]
\centering
\TABLE{Average inference times for the transductive OPF experiments. All times are reported in seconds.
\label{tab:inference_times_transductive}}
{\begin{tabular}{lrrrr}
\toprule
\multirow{2}{*}{\textbf{Case}} &
\multicolumn{2}{c}{\textbf{Proposed Approach}} &
\textbf{Untuned} &
\multirow{2}{*}{\textbf{AC-OPF}} \\
\cmidrule(lr){2-3}
&
\textbf{GNN} &
\textbf{Total} &
\textbf{DC-OPF} &
\\
\midrule
1,354        & 0.097 & 0.244  & 0.162 & 2.202   \\
2,000        & 0.152 & 0.385  & 0.285 & 3.495   \\
4,661        & 0.380 & 1.228  & 0.637 & 10.841  \\
13,659       & 1.256 & 4.256  & 3.194 & 112.240 \\
30,000 (GPU) & 3.475 & 6.071 & 2.595 & 21.570  \\
\bottomrule
\end{tabular}
}{}
\end{table}

\subsection{Optimal Gas Flow}

The optimal gas flow problem determines an economically efficient operating point for a gas network by deciding the amount of gas supplied by each producer, the gas flow through pipes and compressors, the pressure at each junction, and the operating pressure ratio of each compressor. The objective is to minimize the total production cost associated with the gas supplied by the receipts. For the HF model, we use the WP formulation, which is an NLP problem that captures the nonlinear physical constraints of the gas network. For the LF model, we use the linear-relaxation Weymouth-physics (LRWP) formulation, which relaxes the nonlinear WP constraints using polyhedral approximations, including piecewise-linear and McCormick relaxations. Additional details are provided in Table~{\ref{tab:gnn_architecture_gas}} of Appendix~\ref{app:gas}.

\subsubsection{Results}
We train a GNN to predict two parameters that weight quadratic gas-flow terms in the LF objective, guiding solutions of the piecewise-linear and McCormick relaxations toward the HF solution. Training requires approximately 2--4 minutes across the two gas network instances. 
The results demonstrate a consistent advantage of the self-supervised approach over both the semi-supervised and untuned models. As shown in Table~\ref{tab:gas_transductive}, the self-supervised approach achieves mean relative objective errors of only $0.032\%$ and $0.457\%$ for GasLib-40-E and GasLib-135-F, respectively, while also substantially reducing the P90 feasibility violations. In contrast, the semi-supervised approach provides essentially no improvement over the untuned model for GasLib-40-E and only a moderate improvement for GasLib-135-F. 
\begin{table}[H]
\centering
\TABLE{HF evaluation of the transductive OGF experiments. Objective errors are reported as mean (maximum) relative error.
\label{tab:gas_transductive}}
{\begin{tabular}{l|cc|cc|cc}
\hline \up
\multirow{2}{*}{\textbf{Case}} &
\multicolumn{2}{c|}{\textbf{Self-Supervised}} &
\multicolumn{2}{c|}{\textbf{Semi-Supervised}} &
\multicolumn{2}{c}{\textbf{Untuned}} \\
\cline{2-7} \up
&
\begin{tabular}[c]{@{}c@{}}Error (\%)\\Mean (Max)\end{tabular} &
\begin{tabular}[c]{@{}c@{}}P90 Feas.\\(kg/s)\end{tabular} &
\begin{tabular}[c]{@{}c@{}}Error (\%)\\Mean (Max)\end{tabular} &
\begin{tabular}[c]{@{}c@{}}P90 Feas.\\(kg/s)\end{tabular} &
\begin{tabular}[c]{@{}c@{}}Error (\%)\\Mean (Max)\end{tabular} &
\begin{tabular}[c]{@{}c@{}}P90 Feas.\\(kg/s)\end{tabular} \\
\hline \up

GasLib-40-E
& 0.032 (0.541) & 25.572
& 0.764 (1.779) & 96.442
& 0.763 (1.779) & 96.237 \\

GasLib-135-F
& 0.457 (0.792) & 96.534
& 1.965 (3.409) & 254.650
& 2.330 (3.409) & 275.380 \down \\

\hline
\end{tabular}
}{}
\end{table}

Figure~\ref{fig:gasmodels_evaluation_results} further illustrates these improvements across the test samples. These results indicate that directly incorporating HF feedback into the learning objective improves both objective quality and HF feasibility, providing more informative guidance for parameterizing the LF model than minimizing the decision-space error alone.

\begin{figure*}[htbp]
\centering

\begin{overpic}[width=0.72\textwidth]{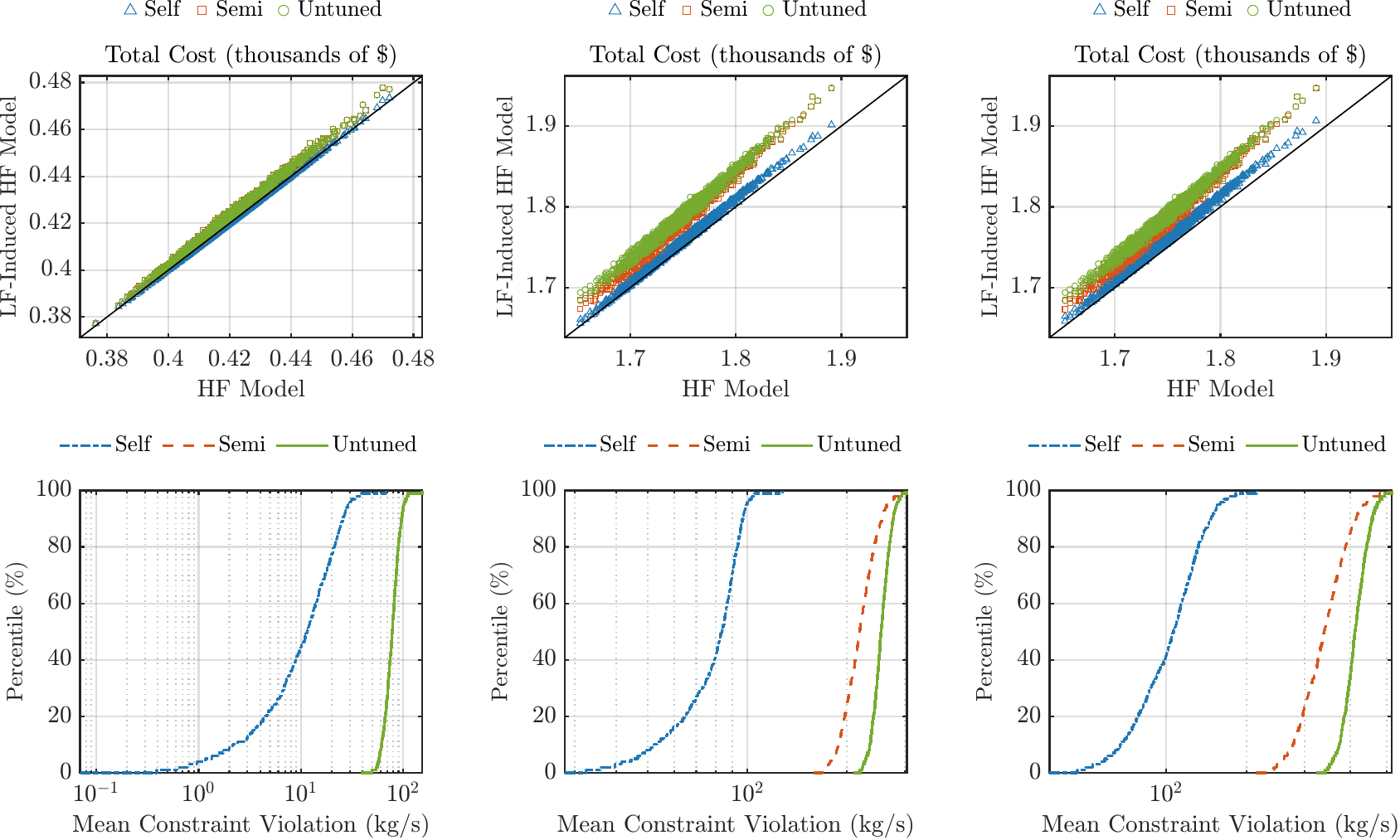}

    \put(16,-2){\makebox(0,0)[c]{\footnotesize (a) GasLib-40E}}
    \put(50,-2){\makebox(0,0)[c]{\footnotesize (b) GasLib-135F}}
    \put(84,-2){\makebox(0,0)[c]{\footnotesize (c) Inductive: 40E to 135F}}

\end{overpic}

\vspace{1.0em}

\caption{Cost and feasibility evaluation for the OGF benchmark instances.
The top panels show cost correlation, and the bottom panels show feasibility
across the test samples.}
\label{fig:gasmodels_evaluation_results}

\end{figure*}

The self-supervised approach also generalizes well in the inductive setting. As shown in Figure~\ref{fig:violin_ogf_inductive}, on GasLib-135-F it achieves a mean relative error of 0.669\%, compared with 1.961\% for the semi-supervised approach. This is only slightly above the transductive self-supervised error of 0.457\% and remains below 1\%, indicating effective generalization to an unseen gas network.

\begin{figure}[ht]
\centering

\includegraphics[
    width=0.2\textwidth,
    trim=0 0 0 25,
    clip
]{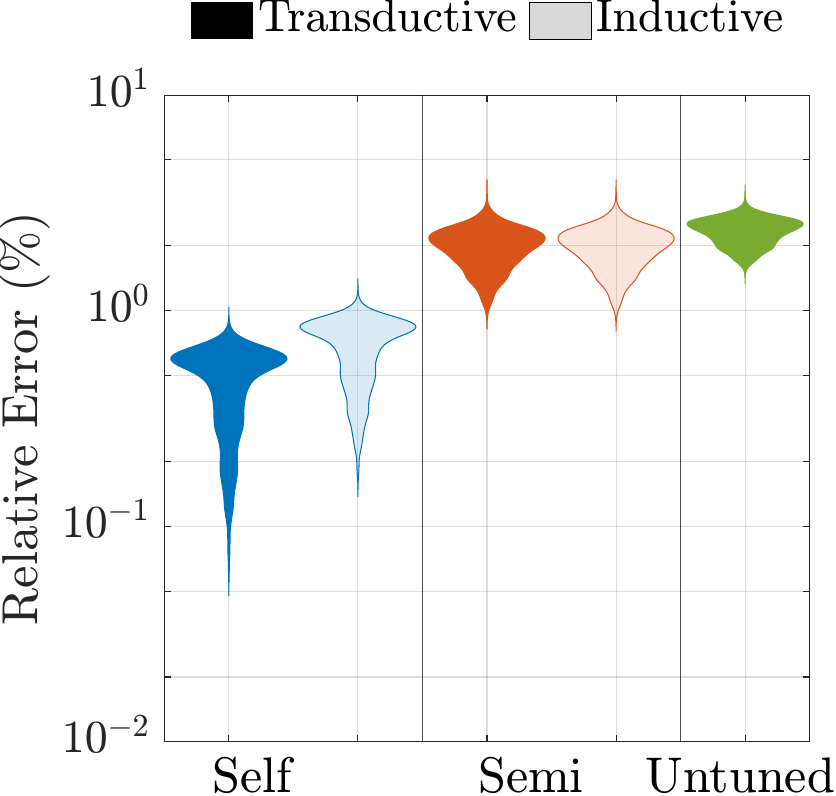}

\caption{Relative objective error for the inductive OGF experiment on GasLib-135-F. Violin plots show the distributions for the self-supervised, semi-supervised, and untuned models.}
\label{fig:violin_ogf_inductive}

\end{figure}

\subsection{The Pooling Problem}

The pooling problem mixes streams with different compositions through intermediate pools to produce final products satisfying quality constraints. We use the pq-formulation \citep{tawarmalani2013convexification}, a nonconvex QCP with bilinear terms $q_{il}y_{lj}$, where $q_{il}$ denotes the fraction of material in pool $l$ originating from source $i$, and $y_{lj}$ denotes the flow from pool $l$ to final product $j$.A common tractable approximation replaces these terms with McCormick relaxations, yielding an LP whose enlarged feasible region may produce solutions that perform poorly under the original bilinear constraints. We add small quadratic penalties that discourage extreme-point solutions while preserving tractability. We refer to the resulting QP as the Quadratic McCormick relaxation and use it as the LF model.

A GNN maps each problem instance to the LF linear objective coefficients. The resulting LF solution \(q\) is then fixed in the HF QCP, which optimizes the remaining flow and slack variables and provides the feedback used for self-supervised training. Test instances of different sizes are generated by perturbing nominal demand and supply by \(\pm10\%\), following a simplified version of \citet{Dey2015}; further details are detailed in Table~\ref{tab:gnn_architecture_pooling} of Appendix~\ref{app:pooling}.

\subsubsection{Results}
As a baseline, we use the \(q\)-variables obtained from the LP McCormick relaxation and evaluate them using the same HF model. Table~\ref{tab:pooling_inference} summarizes the relative objective error and P90 feasibility slack. The self-supervised approach consistently provides the best performance across all five cases, with mean objective errors between \(0.355\%\) and \(2.304\%\) and substantially smaller feasibility violations than the other approaches. The semi-supervised method improves upon the untuned model in Cases 1--4, but its gains are less consistent and it performs worse than the untuned model in Case~5. 

\begin{table}[H]
\centering
\TABLE{HF evaluation of the pooling problem experiments. Objective errors are reported as mean (maximum) relative error.
\label{tab:pooling_inference}}
{\begin{tabular}{l|cc|cc|cc}
\hline \up
\multirow{2}{*}{\textbf{Case}} &
\multicolumn{2}{c|}{\textbf{Self-Supervised}} &
\multicolumn{2}{c|}{\textbf{Semi-Supervised}} &
\multicolumn{2}{c}{\textbf{Untuned}} \\
\cline{2-7} \up
&
\begin{tabular}[c]{@{}c@{}}Error (\%)\\Mean (Max)\end{tabular} &
\begin{tabular}[c]{@{}c@{}}P90 Feas.\\Slack\end{tabular} &
\begin{tabular}[c]{@{}c@{}}Error (\%)\\Mean (Max)\end{tabular} &
\begin{tabular}[c]{@{}c@{}}P90 Feas.\\Slack\end{tabular} &
\begin{tabular}[c]{@{}c@{}}Error (\%)\\Mean (Max)\end{tabular} &
\begin{tabular}[c]{@{}c@{}}P90 Feas.\\Slack\end{tabular} \\
\hline \up

1
& 1.921 (3.328) & $8.897\times10^{-10}$
& 2.500 (4.049) & $4.338\times10^{-3}$
& 5.642 (11.256) & $1.049\times10^{-2}$ \\

2
& 1.004 (2.216) & $1.797\times10^{-4}$
& 8.126 (11.233) & $1.020\times10^{-2}$
& 10.497 (17.250) & $1.388\times10^{-2}$ \\

3
& 1.709 (3.103) & $1.933\times10^{-3}$
& 6.835 (10.936) & $4.291\times10^{-3}$
& 9.943 (14.853) & $8.086\times10^{-3}$ \\

4
& 0.355 (0.719) & $4.387\times10^{-6}$
& 7.578 (11.954) & $3.369\times10^{-3}$
& 8.877 (14.802) & $3.949\times10^{-3}$ \\

5
& 2.304 (4.055) & $6.522\times10^{-4}$
& 18.317 (22.184) & $8.670\times10^{-3}$
& 15.848 (22.832) & $7.582\times10^{-3}$ \down \\

\hline
\end{tabular}
}{}
\end{table}

Figure~\ref{fig:pooling_evaluation_results} further illustrates the cost and feasibility performance across the test instances. The self-supervised approach consistently performs best across all five cases, with mean objective errors below \(2.4\%\) and P90 feasibility slack below \(2\times10^{-3}\). The cost points also remain close to the HF reference across the test samples, while the feasibility distributions are substantially shifted toward zero. These results show the benefit of training directly against HF operational feedback. 
\begin{figure*}[htbp]
\centering

\begin{overpic}[width=\textwidth]{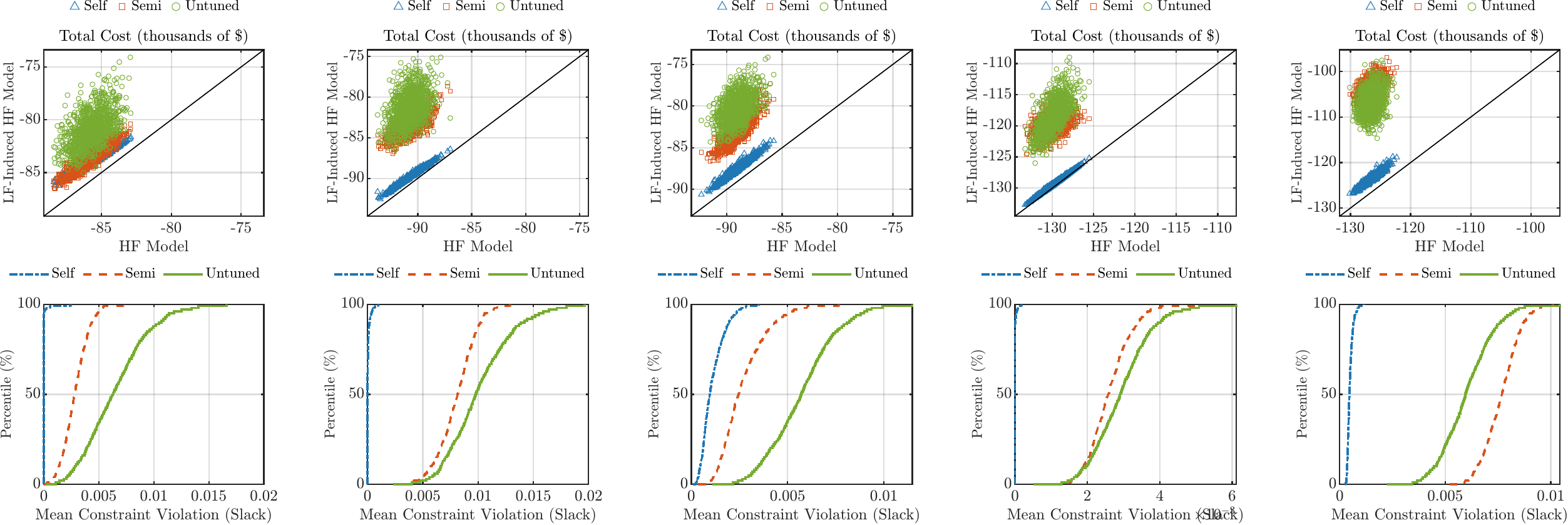}

    \put(10,-2){\makebox(0,0)[c]{\footnotesize (a) Instance 1}}
    \put(30,-2){\makebox(0,0)[c]{\footnotesize (b) Instance 2}}
    \put(50.5,-2){\makebox(0,0)[c]{\footnotesize (c) Instance 3}}
    \put(71.5,-2){\makebox(0,0)[c]{\footnotesize (d) Instance 4}}
    \put(92,-2){\makebox(0,0)[c]{\footnotesize (e) Instance 5}}

\end{overpic}

\vspace{1.0em}

\caption{Cost and feasibility evaluation for the pooling benchmark instances.
The top panels show cost correlation, and the bottom panels show feasibility
across the test samples.}
\label{fig:pooling_evaluation_results}

\end{figure*}

\section{Conclusion}\label{sec:Conclusion}

This paper proposed a self-supervised learning framework for improving tractable low-fidelity optimization models using feedback from high-fidelity formulations. The proposed approach preserves the structure and computational advantages of the low-fidelity model while learning how its parameters should be adjusted as a function of the operating conditions. By embedding the low-fidelity model and the high-fidelity evaluation problem into the training pipeline, the parameters are learned directly with respect to the operational objective of interest, without requiring high-fidelity optimal solutions as training labels.

The computational experiments on optimal power flow, gas network optimization, and pooling problems demonstrate the benefits of this approach across different sources of model mismatch, including approximated nonlinear physics, simplified network models, and relaxations of nonconvex constraints. Across the three applications, self-supervised training consistently improves the high-fidelity performance and feasibility of the solutions obtained from the low-fidelity models. The results also show that the learned parameterizations can generalize to unseen operating conditions and, for the network applications, to systems different from those used during training.

The results show that high-fidelity models can provide useful training signals without being solved to global optimality to generate labels. This is particularly relevant in applications where high-fidelity models are too expensive for repeated online use but can still be evaluated during offline training. Future work will investigate extensions to optimization problems with discrete decisions, stochastic and multistage formulations, and settings in which the selection of the low-fidelity parameters to be tuned is itself learned.
\ACKNOWLEDGMENT{C.L. would like to acknowledge the financial support from the Office of Naval Research (ONR) award
N000142412641 and the ExxonMobil Corporation. 

\noindent The authors affiliated with the University of Colorado Boulder acknowledge support from the startup funding provided by the Department of Electrical, Computer, and Energy Engineering.}

\bibliographystyle{informs2014} \bibliography{sample}

\newpage
\begin{APPENDICES}

\section{Common Computational and Training Details}
\label{app:common}

This appendix provides implementation and experimental details for the optimal power flow, optimal gas flow, and pooling applications considered in Section~\ref{sec:Results}. Across all three applications, the proposed learning framework follows the same general structure. A graph neural network (GNN) maps the operating condition and static problem data to a set of parameters of the low-fidelity (LF) optimization model. The
resulting LF problem is solved to obtain the shared decision variables, which are subsequently evaluated using the corresponding high-fidelity (HF) model.

For self-supervised learning, the HF model is incorporated directly into the training loop. The HF evaluation objective provides the training loss, and sensitivities with respect to the LF decision are obtained from the dual multipliers associated with the LF--HF coupling constraints. These sensitivities are propagated through the differentiable LF optimization layer and subsequently through the GNN.

For semi-supervised learning, the HF evaluation model is not included in the training loop. Instead, HF reference solutions are computed offline, and the GNN and differentiable LF model are trained to minimize the discrepancy between the shared LF and HF decision variables.

\subsection{Differentiation Through the LF Models}

The LF optimization models are differentiated implicitly through their optimality conditions. Given the gradient of the training loss with respect to the LF decision variables, \texttt{DiffOpt.jl} computes the corresponding vector--Jacobian product with respect to the learned LF parameters. The resulting sensitivities are then propagated through the
GNN using automatic differentiation.

To improve numerical robustness of this differentiation step, we
regularize the linear systems solved internally by \texttt{DiffOpt.jl}.
Instead of solving
\[
\mathbf{LHS}\backslash\mathbf{RHS},
\]
we solve
\[
\left(\mathbf{LHS}+\epsilon\mathbf{I}\right)
\backslash\mathbf{RHS},
\qquad
\epsilon=10^{-8}.
\]
This modification avoids occasional
\texttt{LinearAlgebra.SingularException} errors caused by singular or nearly singular systems encountered during differentiation. The regularization parameter is sufficiently small that it improves numerical stability without materially affecting the LF solution or the computed sensitivities.

\subsection{HF Sensitivities}

The HF models are not differentiated through their solver iterations. Instead, the sensitivity of the optimal HF evaluation value with respect to the LF decision is obtained from the dual multipliers of the constraints that couple the LF decision to the HF model, as described in  Section~\ref{sec:self_supervised}. Consequently, the backward pass requires differentiation through the LF optimization problem but not through the internal iterations used to solve the HF NLP or QCP.

\section{Optimal Power Flow}
\label{app:opf}

\subsection{HF and LF Models}

The HF model for the power-system experiments is the AC optimal power flow (AC-OPF), which captures nonlinear AC power-flow equations, voltage magnitudes, reactive power, transmission losses, and the operational limits of generators and transmission elements. The LF model is the DC optimal power flow (DC-OPF), formulated as a convex quadratic program. The DC approximation neglects active-power losses, assumes voltage magnitudes close to one per unit, and linearizes the active-power flow relationships.

The shared decision passed from the LF model to the HF evaluation model is the active-power generation dispatch. Given the dispatch produced by the DC-OPF, the HF evaluation problem imposes
\begin{equation}
P_g
=
\widehat P_g^{\mathrm{DC}}
+
s_g^+
-
s_g^-,
\qquad
s_g^+,s_g^-\geq0,
\end{equation}
where $s_g^+$ and $s_g^-$ represent corrective redispatch required to recover feasibility under the AC model. The HF objective includes the AC generation cost together with a penalty on these corrective actions.

\subsection{Parameterized DC-OPF Model}

The learned parameters modify the active-power demand used by the DC-OPF. For each bus $i$, the GNN predicts a positive scaling coefficient $\beta_i$, and the active demand supplied to the LF model is 
\begin{equation}
\widetilde P_i^d
=
\beta_i P_i^d
+
G_i^{\mathrm{sh}},
\end{equation}
where $P_i^d$ is the original active demand and $G_i^{\mathrm{sh}}$ is
the shunt conductance.

The purpose of this parameterization is to compensate implicitly for effects that are omitted by the DC approximation, particularly active-power losses. Importantly, the GNN does not directly predict the generator dispatch. It predicts parameters of the DC-OPF, and the dispatch remains the solution of the resulting optimization model.

The complete mapping is
\begin{equation}
(X,E,G)
\longrightarrow
\bm{\beta}
\longrightarrow
\operatorname{DCOPF}(\bm{\beta})
\longrightarrow
\widehat{\bm P}_g^{\mathrm{DC}}.
\end{equation}
For self-supervised learning, this is followed by
\begin{equation}
\widehat{\bm P}_g^{\mathrm{DC}}
\longrightarrow
\operatorname{ACOPF\text{-}slack}
\longrightarrow
\Phi^{\mathrm{AC}},
\end{equation}
where $\Phi^{\mathrm{AC}}$ denotes the HF evaluation loss.

\subsection{Benchmark Systems and Instance Generation}

We consider systems ranging from 500 to 30,000 buses. Table \ref{tab:opf_problem_size} reports the dimensions of the corresponding AC-OPF and DC-OPF formulations.

\begin{table}[H]
\centering
\TABLE{Number of variables and constraints for the AC-OPF and DC-OPF formulations.
\label{tab:opf_problem_size}}
{\begin{tabular}{l|cc|cc}
\hline \up
\multirow{2}{*}{\textbf{Case}} &
\multicolumn{2}{c|}{\textbf{AC-OPF}} &
\multicolumn{2}{c}{\textbf{DC-OPF}} \\
\cline{2-5} \up
& Variables & Constraints & Variables & Constraints \\
\hline \up
500    & 4,596   & 6,268   & 1,899  & 2,685 \\
1,354  & 11,192  & 41,813  & 3,605  & 9,558 \\
2,000  & 19,008  & 71,427  & 5,871  & 16,182 \\
4,661  & 34,758  & 130,984 & 11,382 & 29,852 \\
13,659 & 117,370 & 435,770 & 38,218 & 101,870 \\
30,000 & 208,624 & 664,840 & 68,919 & 178,465 \down\\
\hline
\end{tabular}
}{}
\end{table}

Operating conditions are generated by perturbing the nominal demands. Active-power demand varies by up to $\pm30\%$, while reactive-power demand varies between its nominal value and a $15\%$ reduction. At each training iteration, a demand realization is sampled from this distribution. Evaluation is performed on 1,000 unseen operating conditions drawn from the same distribution.

For the inductive experiments, the GNN is trained only on the 500-bus system and applied without retraining to the larger networks.

\subsection{Graph Representation and GNN Architecture}

The GNN follows the transmission-network topology. Buses define the graph nodes and active transmission branches define the graph edges.

Each bus $i$ is represented by
\begin{equation}
\mathbf{x}_i =
\begin{bmatrix}
P_i^{d} &
Q_i^{d} &
V_i^{\max} &
V_i^{\min} &
B_i^{\mathrm{sh}} &
G_i^{\mathrm{sh}}
\end{bmatrix}^{\top},
\end{equation}
where $P_i^{d}$ and $Q_i^{d}$ denote active and reactive demands,
$V_i^{\max}$ and $V_i^{\min}$ are voltage-magnitude limits, and
$B_i^{\mathrm{sh}}$ and $G_i^{\mathrm{sh}}$ are the shunt susceptance
and conductance.

Each branch $k$ is represented by
\begin{equation}
\mathbf{e}_k =
\begin{bmatrix}
r_k &
x_k &
b_k &
S_k^{\max} &
\tau_k
\end{bmatrix}^{\top},
\end{equation}
where $r_k$, $x_k$, $b_k$, $S_k^{\max}$, and $\tau_k$ denote branch resistance, reactance, charging susceptance, thermal limit, and transformer tap ratio.

The GNN contains three edge-aware GATv2 convolutional layers with hidden dimensions 1024, 512, and 256. The resulting bus embeddings are passed through a two-layer decoder with dimensions $256\rightarrow64\rightarrow1$. A Softplus output transformation produces one positive coefficient $\beta_i$ for each bus.

Table~\ref{tab:gnn_architecture_opf} summarizes the architecture.

\begin{table}[t]
\centering
\caption{GNN architecture and optimization-aware pipeline for OPF.}
\label{tab:gnn_architecture_opf}
\small
\setlength{\tabcolsep}{8pt}
\renewcommand{\arraystretch}{1.05}

\begin{tabular*}{\textwidth}{
    @{\extracolsep{\fill}}
    l c c c
}
\hline
\textbf{Stage} &
\textbf{Layer} &
\textbf{Input} &
\textbf{Output} \\
\hline

\multicolumn{4}{c}{\textbf{Graph Representation}} \\
\hline
Node features & Input & -- & $6\times |N|$ \\
Edge features & Input & -- & $5\times |E|$ \\

\hline
\multicolumn{4}{c}{\textbf{GNN Encoder}} \\
\hline
Message passing
& GATv2Conv 1 & $(6,5)$ & $1024\times |N|$ \\
& GATv2Conv 2 & $(1024,5)$ & $512\times |N|$ \\
& GATv2Conv 3 & $(512,5)$ & $256\times |N|$ \\

\hline
\multicolumn{4}{c}{\textbf{Bus-Level Decoder}} \\
\hline
Coefficient decoder
& Dense 1 & $256\times |N|$ & $64\times |N|$ \\
& Dense 2 & $64\times |N|$ & $1\times |N|$ \\
Positive output
& Softplus & $1\times |N|$
& $\bm{\beta}\in\mathbb{R}_{+}^{|N|}$ \\

\hline
\multicolumn{4}{c}{\textbf{Optimization Pipeline}} \\
\hline
Demand parameterization
& Scaling
& $\bm{\beta},\bm P^d,\bm G^{\mathrm{sh}}$
& $\widetilde{\bm P}^d$ \\
Low-fidelity layer
& DC-OPF
& $\widetilde{\bm P}^d$
& $\widehat{\bm P}_{g}^{\mathrm{DC}}$ \\
High-fidelity layer
& AC-OPF-slack
& $\widehat{\bm P}_{g}^{\mathrm{DC}}$
& $\Phi^{\mathrm{AC}}$ \\
\hline
\end{tabular*}
\end{table}

\subsection{Training Procedure}

\paragraph{Initialization.}
Training begins with a coefficient-initialization stage in which the GNN is trained for 100 epochs to produce $\beta_i=1$ for every bus. AdamW is used with a learning rate of $10^{-4}$. This initialization corresponds to the original untuned DC-OPF.

\paragraph{Self-supervised training.}
The initialized GNN is embedded in the complete LF--HF training pipeline. The training loss is the HF evaluation objective,
\begin{equation}
\mathcal L_{\mathrm{self}}
=
\Phi^{\mathrm{AC}}.
\end{equation}
The DC-OPF layer is differentiated using \texttt{DiffOpt.jl}. The sensitivities of the HF evaluation value with respect to the DC-OPF dispatch are obtained from the dual multipliers of the generator-linking constraints.

\paragraph{Semi-supervised training.}
The HF evaluation layer is omitted. Instead, the DC-OPF dispatch is
compared with a previously computed AC-OPF dispatch,
\begin{equation}
\mathcal L_{\mathrm{semi}}
=
\operatorname{MSE}
\left(
\widehat{\bm P}_g^{\mathrm{DC}},
\bm P_g^{\mathrm{AC}}
\right).
\end{equation}

Both approaches use AdamW with a learning rate of $10^{-6}$, a
mini-batch size of one, and random seed 20. Training proceeds for at most 5,000 epochs and is also subject to the time limits reported in Table~\ref{tab:training_times}. For model selection, the checkpoint with the lowest mean relative HF objective error is retained.

\subsection{Large-Scale 30,000-Bus Implementation}

For the 30,000-bus system, the HF AC-OPF correction model is implemented using \texttt{ExaModels.jl} and solved on a GPU with \texttt{MadNLP.jl}. The backward pass does not differentiate through the iterations of \texttt{MadNLP}. Instead, the dual multipliers associated with
\begin{equation}
P_g-\widehat P_g^{\mathrm{DC}}-s_g^++s_g^-=0
\end{equation}
provide the required gradient of the HF evaluation objective with respect to the LF dispatch.

This allows the HF model to benefit from GPU acceleration while preserving the same self-supervised training principle used for the smaller systems.

\subsection{Training Time}

For a fair comparison, self-supervised and semi-supervised training are performed under the same time budget for each system. The semi-supervised model generally completes more epochs because it does not solve the HF model during training. Table~\ref{tab:training_times} summarizes the computational cost.

\begin{table}[H]
\centering
\TABLE{Training time summary for the OPF experiments. Forward and backward times are reported in minutes.
\label{tab:training_times}}
{\begin{tabular}{l|c|cc|cc|c|cc}
\hline
\multirow{2}{*}{\textbf{Case}} &
\multirow{2}{*}{\textbf{Method}} &
\multicolumn{2}{c|}{\textbf{AC}} &
\multicolumn{2}{c|}{\textbf{DC}} &
\textbf{Time Limit} &
\multicolumn{2}{c}{\textbf{Stopping}} \\
\cline{3-4}
\cline{5-6}
\cline{8-9}
& & Forward & Backward & Forward & Backward & (min) & Best Epoch & Stop \\
\hline

\multirow{2}{*}{500}
& Self & 2.62 & $\approx0$ & 0.51 & 0.82 & 5 & 78 & 477 \\
& Semi & -- & -- & 1.05 & 1.62 & 5 & 20 & 1156 \\
\hline

\multirow{2}{*}{1,354}
& Self & 11.97 & $\approx0$ & 1.74 & 12.80 & 30 & 35 & 684 \\
& Semi & -- & -- & 2.88 & 21.90 & 30 & 352 & 1174 \\
\hline

\multirow{2}{*}{2,000}
& Self & 65.63 & $\approx0$ & 7.28 & 23.54 & 120 & 1788 & 1795 \\
& Semi & -- & -- & 15.20 & 48.44 & 120 & 3557 & 3722 \\
\hline

\multirow{2}{*}{4,661}
& Self & 37.88 & $\approx0$ & 6.98 & 64.93 & 120 & 20 & 477 \\
& Semi & -- & -- & 10.95 & 96.14 & 120 & 10 & 835 \\
\hline

\multirow{2}{*}{13,659}
& Self & 79.82 & $\approx0$ & 11.58 & 59.69 & 180 & 53 & 221 \\
& Semi & -- & -- & 22.61 & 130.39 & 180 & 34 & 429 \\
\hline

\multirow{2}{*}{30,000 (GPU)}
& Self & 69.43 & $\approx0$ & 12.10 & 130.10 & 240 & 99 & 131 \\
& Semi & -- & -- & 15.18 & 190.59 & 240 & 12 & 159 \\
\hline
\end{tabular}
}{}
\end{table}

\section{Optimal Gas Flow}
\label{app:gas}

\subsection{HF and LF Models}

The HF gas-network model is the nonlinear WP formulation implemented in \texttt{GasModels.jl}. It represents the nonlinear relationship between gas pressure and flow together with compressor operation, receipt injections, withdrawals, and pressure limits.

The LF model is the LRWP formulation, which replaces selected nonlinear relationships with polyhedral and McCormick relaxations. The resulting model is significantly easier to solve but may select receipt injections that are suboptimal or incompatible with the original nonlinear gas-flow equations.

The shared decisions between the LF and HF models are the receipt injections. Given the LF receipt vector
$\widehat{\bm g}^{\mathrm{LF}}$, the HF evaluation model uses
\begin{equation}
g_r
=
\widehat g_r^{\mathrm{LF}}
+
s_r^+
-
s_r^-,
\qquad
s_r^+,s_r^-\geq0.
\end{equation}
The slack variables quantify the adjustments required to make the LF decision compatible with the nonlinear HF model.

\subsection{Parameterized LRWP Model}

Rather than modifying the gas-network constraints, the GNN modifies the objective of the LRWP model. It produces two classes of positive coefficients:
\begin{itemize}
    \item junction-level coefficients $\beta_i$, which weight quadratic pressure-related terms; and
    \item pipe-level coefficients $\alpha_k$, which weight quadratic flow-related terms.
\end{itemize}

The LF objective therefore contains the nominal receipt-production cost together with learned quadratic terms of the general form
\begin{equation}
\sum_{i\in J}\beta_i \pi_i^2
+
\sum_{k\in P}\alpha_k f_k^2.
\end{equation}

These terms do not attempt to reconstruct the nonlinear gas equations explicitly. Instead, they reshape the objective of the tractable LRWP model so that its optimizer produces receipt decisions that perform well when evaluated using the nonlinear WP model.

The corresponding self-supervised computational pipeline is
\begin{equation}
(X,E,G)
\longrightarrow
(\bm{\beta},\bm{\alpha})
\longrightarrow
\operatorname{LRWP}(\bm{\beta},\bm{\alpha})
\longrightarrow
\widehat{\bm g}^{\mathrm{LF}}
\longrightarrow
\operatorname{WP\text{-}slack}
\longrightarrow
\Phi^{\mathrm{HF}}.
\end{equation}

\subsection{Benchmark Systems and Instance Generation}

We use the GasLib-40-E and GasLib-135-F benchmark networks. Their problem sizes are reported in Table~\ref{tab:process_problem_size_gas}.

\begin{table}[H]
\centering
\TABLE{Number of variables and constraints for the OGF benchmark problems.
\label{tab:process_problem_size_gas}}
{\begin{tabular}{llrrrr}
\toprule
\multirow{2}{*}{\textbf{Case}} &
\multirow{2}{*}{\textbf{Model}} &
\multirow{2}{*}{\textbf{Variables}} &
\multicolumn{3}{c}{\textbf{Constraints}} \\
\cmidrule(lr){4-6}
& & &
\textbf{Linear} &
\textbf{Nonlinear} &
\textbf{Bounds} \\
\midrule
\multirow{2}{*}{GasLib-40-E}
& WP   & 101   & 131   & 57  & 200 \\
& LRWP & 307   & 346   & 0   & 524 \\
\midrule
\multirow{2}{*}{GasLib-135-F}
& WP   & 347   & 476   & 228 & 692 \\
& LRWP & 1,109 & 1,329 & 0   & 1,878 \\
\bottomrule
\end{tabular}
}{}
\end{table}

For each gas network, 1,000 realizations are generated. Training uses one realization, while the remaining operating conditions are used to assess generalization.

The original benchmark data are modified to produce a sufficiently meaningful LF--HF mismatch for evaluating the proposed learning method. The modifications include changes to selected compressor ratios and pipe diameters together with the addition of receipt locations. These changes increase the difference between the relaxed and nonlinear formulations without altering the structure of the learning framework.

For the inductive experiment, the model trained on GasLib-40-E is applied directly to GasLib-135-F.

\subsection{Graph Representation and GNN Architecture}

Junctions define graph nodes and physical pipes define graph edges.

Each junction $i$ has the feature vector
\begin{equation}
\mathbf{x}_i =
\begin{bmatrix}
L_i &
P_i^U &
P_i^L &
T_i
\end{bmatrix}^{\top},
\end{equation}
where $L_i$ is the sampled withdrawal, $P_i^U$ and $P_i^L$ are the upper and lower pressure limits, and $T_i$ is the GasModels junction-type indicator.

Each pipe $k$ has the feature vector
\begin{equation}
\mathbf{e}_k =
\begin{bmatrix}
D_k &
L_k^{\mathrm{pipe}} &
F_k^{\mathrm{fric}}
\end{bmatrix}^{\top},
\end{equation}
where $D_k$, $L_k^{\mathrm{pipe}}$, and $F_k^{\mathrm{fric}}$ denote pipe diameter, length, and friction factor. The direction of each graph edge follows the stored from-junction and to-junction orientation.

The GNN contains three edge-aware GATv2 layers with five attention heads. The layers contain 512, 256, and 128 features per head, producing concatenated representations with dimensions 2560, 1280, and 640. The final junction embeddings are processed by two output decoders. A junction-level decoder produces $\beta_i$, while a pipe-level decoder uses the concatenated embeddings of the incident junctions to produce $\alpha_k$. Both outputs use Softplus transformations to ensure
positivity.

\begin{table}[t]
\centering
\caption{GNN architecture and optimization-aware pipeline for the gas-flow problem.}
\label{tab:gnn_architecture_gas}
\small
\setlength{\tabcolsep}{6pt}
\renewcommand{\arraystretch}{1.05}

\begin{tabular*}{\textwidth}{
    @{\extracolsep{\fill}}
    l c c c
}
\hline
\textbf{Stage} &
\textbf{Layer} &
\textbf{Input} &
\textbf{Output} \\
\hline

\multicolumn{4}{c}{\textbf{Graph Representation}} \\
\hline
Node features & Input & -- & $4\times |J|$ \\
Edge features & Input & -- & $3\times |P|$ \\

\hline
\multicolumn{4}{c}{\textbf{GNN Encoder (5 Attention Heads)}} \\
\hline
Message passing
& GATv2Conv 1 & $(4,3)$ & $512\times5=2560$ \\
& GATv2Conv 2 & $(2560,3)$ & $256\times5=1280$ \\
& GATv2Conv 3 & $(1280,3)$ & $128\times5=640$ \\

\hline
\multicolumn{4}{c}{\textbf{Junction-Level Decoder}} \\
\hline
Coefficient decoder
& Dense 1 & $640\times |J|$ & $64\times |J|$ \\
& Dense 2 & $64\times |J|$ & $1\times |J|$ \\
Positive output
& Softplus & $1\times |J|$
& $\bm{\beta}\in\mathbb{R}_{+}^{|J|}$ \\

\hline
\multicolumn{4}{c}{\textbf{Pipe-Level Decoder}} \\
\hline
Pipe representation
& Concatenation & $640+640$ & $1280$ \\
Coefficient decoder
& Dense 1 & $1280\times |P|$ & $64\times |P|$ \\
& Dense 2 & $64\times |P|$ & $1\times |P|$ \\
Positive output
& Softplus & $1\times |P|$
& $\bm{\alpha}\in\mathbb{R}_{+}^{|P|}$ \\

\hline
\multicolumn{4}{c}{\textbf{Optimization Pipeline}} \\
\hline
Low-fidelity layer
& LRWP
& $\bm{\alpha},\bm{\beta}$
& $\widehat{\bm g}^{\mathrm{LF}}$ \\
High-fidelity layer
& WP-slack
& $\widehat{\bm g}^{\mathrm{LF}}$
& $\Phi^{\mathrm{HF}}$ \\
\hline
\end{tabular*}
\end{table}

\subsection{Training Procedure}

\paragraph{Initialization.}
The GNN is first trained for 100 epochs to produce coefficient values equal to one. AdamW is used with a learning rate of $10^{-4}$. This initialization reproduces the nominal LRWP objective before self-supervised or semi-supervised training.

\paragraph{Self-supervised training.}
The GNN predicts $\bm{\alpha}$ and $\bm{\beta}$, which parameterize the LRWP problem and produce the receipt-injection vector $\widehat{\bm g}^{\mathrm{LF}}$. The HF evaluation loss is
\begin{equation}
\mathcal L_{\mathrm{self}}
=
\sum_{r\in R} K_r g_r
+
\sum_{r\in R}
K_r
\left(s_r^+ + s_r^-\right)^2.
\end{equation}
The LRWP layer is differentiated with \texttt{DiffOpt.jl}, while the dual multipliers of the HF receipt-linking constraints provide the sensitivity with respect to the LF receipt vector.

\paragraph{Semi-supervised training.}
The HF evaluation layer is omitted. Instead, the predicted receipt vector is compared with the HF reference receipt injections:
\begin{equation}
\mathcal L_{\mathrm{semi}}
=
\operatorname{MSE}
\left(
\frac{
\widehat{\bm g}^{\mathrm{LF}}
}{
\max(\bm G^{\mathrm{HF}},10^{-7})
},
\bm 1
\right),
\end{equation}
where the maximum and division operations are componentwise.

Both approaches use AdamW with a learning rate of $10^{-5}$, a
mini-batch size of one, and random seed 20.

\subsection{Training Time}

The gas-network models are substantially smaller than the OPF and pooling instances, allowing training to be completed in a few minutes. The self-supervised and semi-supervised approaches are given the same training time for each network.

\begin{table}[H]
\centering
\TABLE{Training time summary for the OGF experiments. Times are reported in minutes.
\label{tab:training_times_summary_ogf}}
{\begin{tabular}{l|c|c|cc|cc|c}
\hline
\multirow{2}{*}{\textbf{Case}} &
\multirow{2}{*}{\textbf{Method}} &
\multirow{2}{*}{\textbf{Forward Passes}} &
\multicolumn{2}{c|}{\textbf{LF}} &
\multicolumn{2}{c|}{\textbf{HF}} &
\multirow{2}{*}{\textbf{Time Limit}} \\
\cline{4-5}
\cline{6-7}
& & & Forward & Backward & Forward & Backward & \\
\hline

\multirow{2}{*}{GasLib-40-E}
& Self & 1209 & 0.389 & 0.195 & 0.787 & $1.50\times10^{-4}$ & 2 \\
& Semi & 1809 & 0.687 & 0.248 & -- & -- & 2 \\
\hline

\multirow{2}{*}{GasLib-135-F}
& Self & 665 & 0.863 & 0.720 & 1.728 & $1.67\times10^{-4}$ & 4 \\
& Semi & 1083 & 1.608 & 0.949 & -- & -- & 4 \\
\hline
\end{tabular}
}{}
\end{table}

\section{Pooling Problem}
\label{app:pooling}

\subsection{HF and LF Models}

The HF pooling model is the original nonconvex QCP, which contains bilinear relationships between the pooling variables. The LF model is constructed from the McCormick relaxation of these bilinear terms, augmented with small quadratic terms that reduce the tendency of the relaxation to select undesirable extreme-point solutions. We refer to this model as the Quadratic McCormick relaxation.

The shared quantities between the LF and HF models are the $q$-variables. After solving the parameterized LF problem, the resulting $\widehat{\bm q}$ is fixed in the HF model. The remaining HF variables are optimized subject to the original bilinear relationships, with slack variables included where necessary to ensure that the HF evaluation problem remains well defined.

\subsection{Parameterized Quadratic McCormick Model}

The GNN predicts four families of LF objective coefficients,
\begin{equation}
(\bm{\alpha},\bm{\beta},\bm{\sigma},\bm{\eta}),
\end{equation}
associated with the different node-, edge-, and path-level terms of the Quadratic McCormick model.

The coefficients modify the objective of the tractable LF problem rather than directly predicting the pooling decisions. The learned objective therefore changes which point of the relaxed feasible region is selected while retaining the same LF constraint structure.

The GNN output is used according to
\begin{equation}
(X,E,G)
\longrightarrow
H
\longrightarrow
(\bm{\alpha},\bm{\beta},\bm{\sigma},\bm{\eta})
\longrightarrow
\operatorname{Quadratic\ McCormick}
\longrightarrow
\widehat{\bm q}.
\end{equation}

For self-supervised learning, the LF solution is subsequently evaluated using the HF pooling model,
\begin{equation}
\widehat{\bm q}
\longrightarrow
\operatorname{HF\text{-}slack}
\longrightarrow
\Phi^{\mathrm{HF}}.
\end{equation}

\subsection{Benchmark Systems and Instance Generation}

We consider five synthetic pooling instances of increasing size. Table \ref{tab:problem_size_pooling} summarizes the number of input streams, pools, output products, specifications, variables, and constraints.

\begin{table}[H]
\centering
\TABLE{Problem sizes for the pooling benchmark instances.
\label{tab:problem_size_pooling}}
{\begin{tabular}{l|cccc|ccc|cc}
\hline \up
\multirow{3}{*}{\textbf{Case}} &
\multicolumn{4}{c|}{\textbf{Size}} &
\multicolumn{3}{c|}{\textbf{High-Fidelity}} &
\multicolumn{2}{c}{\textbf{Low-Fidelity}} \\
\cline{2-10} \up
& \multirow{2}{*}{Inputs}
& \multirow{2}{*}{Pools}
& \multirow{2}{*}{Outputs}
& \multirow{2}{*}{Specs}
& \multirow{2}{*}{Variables}
& \multicolumn{2}{c|}{Constraints}
& \multirow{2}{*}{Variables}
& \multirow{2}{*}{Constraints} \\
\cline{7-8} \up
& & & & & & Linear & Quadratic & & \\
\hline \up
1 & 25 & 18 & 25 & 8  & 2,542  & 1,096 & 2,135  & 2,542  & 7,501 \\
2 & 25 & 22 & 30 & 10 & 3,189  & 1,474 & 2,676  & 3,189  & 9,502 \\
3 & 30 & 22 & 35 & 10 & 4,391  & 1,733 & 3,770  & 4,391  & 13,043 \\
4 & 40 & 30 & 45 & 10 & 10,127 & 2,706 & 9,015  & 10,127 & 29,751 \\
5 & 40 & 30 & 50 & 14 & 11,273 & 3,355 & 10,083 & 11,273 & 33,604 \down\\
\hline
\end{tabular}
}{}
\end{table}

Test instances are generated by perturbing nominal supply and demand values by $\pm10\%$. During the initial training stage, the nominal instance is used. Fine-tuning subsequently uses the nominal instance together with 24 additional realizations from the target distribution. Performance is evaluated on unseen realizations generated from the same distribution.

\subsection{Graph Representation and GNN Architecture}

The pooling network is represented as a heterogeneous graph. Sources $I$, intermediate pools $L$, and terminals $J$ define three node types, while source--pool, pool--terminal, and source--terminal connections define three edge relations.

Each node $v$ is represented by
\begin{equation}
\mathbf{x}_v =
\begin{bmatrix}
\mathbf t_v &
c_v
\end{bmatrix}^{\top},
\end{equation}
where $\mathbf t_v$ is a three-dimensional one-hot encoding identifying the node type and $c_v$ denotes the corresponding capacity or demand.

Each edge $(u,v)$ is represented by
\begin{equation}
\mathbf{e}_{uv} =
\begin{bmatrix}
C_{uv} &
U_{uv} &
\mathbf r_{uv}
\end{bmatrix}^{\top},
\end{equation}
where $C_{uv}$ is the transportation cost, $U_{uv}$ is the corresponding flow bound, and $\mathbf r_{uv}$ is a three-dimensional one-hot encoding of the edge relation.

Continuous node and edge features are standardized using statistics computed from a nominal reference instance. Node normalization is performed separately for sources, pools, and terminals. Categorical one-hot features are not normalized.

The GNN first maps the node features to a 32-dimensional latent representation. Three heterogeneous edge-aware GATv2 layers then produce representations with dimensions 128, 64, and 32 using four attention heads. Separate message-passing operators are used for the three relation types, and their contributions are aggregated by summation. Residual connections are included at each layer.

The learned node embeddings are then processed by edge- and triplet-level decoders. For an edge $(u,v)$, the source and destination embeddings are concatenated with the edge features and mapped to the coefficients $\eta$, $\beta$, and $\sigma$. For each source--pool--terminal triplet $(i,l,j)$, the embeddings of the three nodes are concatenated and mapped
to the coefficient $\alpha$.

\begin{table}[t]
\centering
\caption{GNN architecture and optimization-aware pipeline for the pooling problem.}
\label{tab:gnn_architecture_pooling}
\small
\setlength{\tabcolsep}{7pt}
\renewcommand{\arraystretch}{1.08}

\begin{tabular}{llrr}
\toprule
\textbf{Component} &
\textbf{Operation} &
\textbf{Input} &
\textbf{Output} \\
\midrule

\multicolumn{4}{l}{\textbf{Graph representation}} \\
Node features
& Input
& --
& $4\times |V|$ \\
Edge features
& Input
& --
& $5\times |E|$ \\

\midrule
\multicolumn{4}{l}{\textbf{Input embedding}} \\
Node embedding
& Dense
& $4$
& $32$ \\

\midrule
\multicolumn{4}{l}{\textbf{GNN encoder}} \\
\multirow{3}{*}{Message passing}
& GATv2Conv 1
& $(32,5)$
& $128$ \\
& GATv2Conv 2
& $(128,5)$
& $64$ \\
& GATv2Conv 3
& $(64,5)$
& $32$ \\
Output embedding
& Dense
& $32$
& $32$ \\

\midrule
\multicolumn{4}{l}{\textbf{Edge-level decoders}} \\
Edge representation
& Concatenation
& $32+32+5$
& $69$ \\
\multirow{2}{*}{Coefficient decoder}
& Dense 1
& $69$
& $128$ \\
& Dense 2
& $128$
& $1$ \\
Decoded coefficients
& --
& --
& $\eta,\beta,\sigma$ \\

\midrule
\multicolumn{4}{l}{\textbf{Triplet-level decoder}} \\
Triplet representation
& Concatenation
& $32+32+32$
& $96$ \\
\multirow{2}{*}{Coefficient decoder}
& Dense 1
& $96$
& $128$ \\
& Dense 2
& $128$
& $1$ \\
Decoded coefficient
& --
& --
& $\alpha$ \\

\midrule
\multicolumn{4}{l}{\textbf{Optimization-aware output}} \\
Tunable LF model
& Optimization layer
& $\alpha,\beta,\sigma,\eta$
& $\widehat{\bm q}$ \\
\bottomrule
\end{tabular}
\end{table}

\subsection{Training Procedure}

\paragraph{Initialization.}
The GNN is first trained to produce coefficient values close to one by minimizing the MSE between the predicted coefficients and a constant target of one. AdamW is used with a learning rate of $10^{-2}$, and the initialization stage terminates when the loss falls below $5\times10^{-3}$.

\paragraph{Self-supervised training.}
Training proceeds in three stages. First, the nominal instance is used for 100 epochs with the HF objective obtained from the conservative slack-augmented pooling model. AdamW is used with a learning rate of $10^{-3}$ and gradient clipping with norm 1.0. The resulting model is then trained for another 100 epochs using the aggressive HF evaluation model. The optimizer settings remain unchanged.

Finally, the model is fine-tuned using the nominal instance and 24 additional samples from the target distribution. Mini-batches contain five instances, and the epoch loss is the average HF evaluation loss across the mini-batches. Training continues subject to the computational budget reported below.

\paragraph{Semi-supervised training.}
For each training instance, the full-space QCP is solved with a
500-second time limit to generate a high-quality target value of
$\bm q$. Following initialization, the nominal instance is used to train
the GNN for 200 epochs by minimizing
\begin{equation}
\mathcal L_{\mathrm{semi}}
=
\operatorname{MSE}
\left(
\widehat{\bm q},
\bm q^{\mathrm{HF}}
\right).
\end{equation}

The model is then fine-tuned using the same nominal instance and 24 additional samples used by the self-supervised approach. Mini-batches contain five instances, and the epoch loss is the average MSE across the mini-batches. AdamW is used with a learning rate of $10^{-3}$ and gradient clipping with norm 1.0.

\subsection{Training Time}

The pooling experiments are substantially more computationally demanding than the gas experiments because both the LF and HF optimization models increase rapidly in size. For each benchmark instance, the self-supervised and semi-supervised approaches are assigned the same total training-time budget.

\begin{table}[H]
\centering
\TABLE{Training time summary for the pooling experiments. Times are reported in minutes.
\label{tab:training_times_summary_pooling}}
{\begin{tabular}{l|c|c|cc|cc|c}
\hline
\multirow{2}{*}{\textbf{Case}} &
\multirow{2}{*}{\textbf{Method}} &
\multirow{2}{*}{\textbf{Forward Passes}} &
\multicolumn{2}{c|}{\textbf{LF}} &
\multicolumn{2}{c|}{\textbf{HF}} &
\multirow{2}{*}{\textbf{Time Limit}} \\
\cline{4-5}
\cline{6-7}
& & & Forward & Backward & Forward & Backward & \\
\hline

\multirow{2}{*}{1}
& Self & 450 & 4.9 & 3.5 & 3.3 & 6.9 & 20 \\
& Semi & 925 & 9.2 & 6.6 & -- & -- & 20 \\
\hline

\multirow{2}{*}{2}
& Self & 900 & 14.1 & 9.2 & 9.3 & 23.2 & 60 \\
& Semi & 2050 & 28.7 & 19.9 & -- & -- & 60 \\
\hline

\multirow{2}{*}{3}
& Self & 525 & 11.2 & 10.4 & 7.9 & 29.9 & 60 \\
& Semi & 1300 & 28.2 & 22.3 & -- & -- & 60 \\
\hline

\multirow{2}{*}{4}
& Self & 725 & 53.4 & 40.2 & 43.8 & 159.3 & 300 \\
& Semi & 2050 & 145.1 & 116.1 & -- & -- & 300 \\
\hline

\multirow{2}{*}{5}
& Self & 900 & 97.6 & 81.7 & 74.1 & 337.3 & 600 \\
& Semi & 2550 & 270.2 & 279.2 & -- & -- & 600 \\
\hline
\end{tabular}
}{}
\end{table}

\input{formulation}

\end{APPENDICES}

\end{document}

%% file: Formulation.tex
% \documentclass{article}
% \usepackage[hidelinks]{hyperref}
% \usepackage{amsmath}
% \usepackage{graphicx}
% \usepackage{biblatex}
% \addbibresource{dronaref.bib}
% \newcommand{\mc}[1]{\mathcal{#1}}
% \usepackage{booktabs}
% \usepackage{subcaption}
% \usepackage{algorithm}
% \usepackage{algpseudocode}
% \usepackage[margin=1in]{geometry}
% \usepackage{setspace}
% \usepackage{float}
% \usepackage{xcolor}
% \usepackage{bm}

\title{Formulations}

% \begin{document}

\section{Mathematical Formulation of the Optimal Power Flow Problem}

\subsection{Notation}

\subsubsection{Indices}

\begin{longtable}{@{}p{0.34\linewidth}p{0.60\linewidth}@{}}
\toprule
\textbf{Index} & \textbf{Description} \\
\midrule
$i,j$ & Buses \\
$g$ & Generators \\
\bottomrule
\end{longtable}

\subsubsection{Sets}

\begin{longtable}{@{}p{0.34\linewidth}p{0.60\linewidth}@{}}
\toprule
\textbf{Set} & \textbf{Description} \\
\midrule
$B$ & Set of buses \\
$G$ & Set of generators \\
$E$ & Set of branches represented by bus pairs $(i,j)$ \\
$R$ & Set of reference buses \\
$G_i$ & Set of generators connected to bus $i$ \\
$E_i$ & Set of branches connected to bus $i$ \\
\bottomrule
\end{longtable}

\subsubsection{Parameters}

\begin{longtable}{@{}p{0.34\linewidth}p{0.60\linewidth}@{}}
\toprule
\textbf{Parameter} & \textbf{Description} \\
\midrule
\endfirsthead

\toprule
\textbf{Parameter} & \textbf{Description} \\
\midrule
\endhead

$P_i^D$ & Active-power demand at bus $i$ \\
$Q_i^D$ & Reactive-power demand at bus $i$ \\
$V_i^L$ & Lower voltage-magnitude bound at bus $i$ \\
$V_i^U$ & Upper voltage-magnitude bound at bus $i$ \\
$P_g^L$ & Lower active-power bound of generator $g$ \\
$P_g^U$ & Upper active-power bound of generator $g$ \\
$Q_g^L$ & Lower reactive-power bound of generator $g$ \\
$Q_g^U$ & Upper reactive-power bound of generator $g$ \\
$C_g^{(2)}, C_g^{(1)}, C_g^{(0)}$
& Quadratic, linear, and constant cost coefficients of generator $g$ \\
$C^{\max,1}:=\max_{g\in G}C_g^{(1)}$
& Largest linear generator-cost coefficient \\
$G_{ij}$ & Series conductance of branch $(i,j)$ \\
$B_{ij}$ & Series susceptance of branch $(i,j)$ \\
$B_{ij}^{\mathrm{sh}}$ & Shunt susceptance of transmission line between buses $i$ and $j$ \\
$B_i^{\mathrm{sh}}$ & Shunt susceptance at bus $i$ \\
$G_i^{\mathrm{sh}}$ & Shunt conductance at bus $i$ \\
$T_{ij}$ & Transformer tap ratio of branch $(i,j)$; $T_{ij}=1$ for a branch without a transformer \\
$S_{ij}^U$ & Apparent-power rating of branch $(i,j)$ \\
$\Delta_{ij}^L$ & Lower voltage-angle-difference bound of branch $(i,j)$ \\
$\Delta_{ij}^U$ & Upper voltage-angle-difference bound of branch $(i,j)$ \\
$\alpha_i^{\mathrm{GNN}}$ & GNN coefficient for bus $i$ \\
$\widehat{P}^{\mathrm{LF}}_g$
& Active-power dispatch obtained from the low-fidelity model and passed to the final high-fidelity layer \\
\bottomrule
\end{longtable}

\subsubsection{Variables}

\begin{longtable}{@{}p{0.34\linewidth}p{0.60\linewidth}@{}}
\toprule
\textbf{Variable} & \textbf{Description} \\
\midrule
$p_g$ & Active-power output of generator $g$ \\
$q_g$ & Reactive-power output of generator $g$ \\
$v_i^r$ & Real component of the complex voltage at bus $i$ \\
$v_i^i$ & Imaginary component of the complex voltage at bus $i$ \\
$p_{ij}$ & Active-power flow from bus $i$ to bus $j$ \\
$q_{ij}$ & Reactive-power flow from bus $i$ to bus $j$ \\
$\theta_i$ & Voltage angle at bus $i$ \\
$z_i$ & Nonnegative active-power-balance slack of the low-fidelity model \\
$s_g^+$ & Positive deviation from the low-fidelity active-power dispatch \\
$s_g^-$ & Negative deviation from the low-fidelity active-power dispatch \\
\bottomrule
\end{longtable}

\subsubsection{Rectangular-voltage products}

To write the AC equations compactly, define

\begin{equation}
c_{ii}:=(v_i^r)^2+(v_i^i)^2
\quad \forall i\in B,
\label{eq:voltage-product-diagonal}
\end{equation}

\begin{equation}
c_{ij}:=v_i^r v_j^r+v_i^i v_j^i
\quad \forall (i,j)\in E,
\label{eq:voltage-product-real}
\end{equation}

\begin{equation}
s_{ij}:=v_i^r v_j^i-v_j^r v_i^i
\quad \forall (i,j)\in E.
\label{eq:voltage-product-imaginary}
\end{equation}

These quantities are auxiliary expressions of the rectangular voltages, not
independent optimization variables.

The optimization-variable sets used by the three models are

\begin{equation}
\begin{aligned}
\Xi^{\mathrm{AC}}
:=\bigl\{
&p_g,q_g,
v_i^r,v_i^i,
p_{ij},q_{ij},p_{ji},q_{ji}
% :\\
% &g\in G,\ i\in B,\ (i,j)\in E
% \bigr\},
\}
\end{aligned}
\label{eq:ac-variable-set}
\end{equation}

\begin{equation}
\Xi^{\mathrm{DC}}
:=\left\{
p_g,\theta_i,p_{ij},z_i
% :
% g\in G,\ i\in B,\ (i,j)\in E
\right\}
\label{eq:dc-variable-set}
\end{equation}

% and

\begin{equation}
\Xi^{\mathrm{AC-S}}
:=\Xi^{\mathrm{AC}}
\cup\left\{s_g^+,s_g^-\right\}.
%:g\in G
\label{eq:ac-slack-variable-set}
\end{equation}

% The high-fidelity formulations therefore optimize $v_i^r$ and $v_i^i$
% directly; the three voltage-product expressions above are evaluated from those
% variables.

\subsection{Formulation}

\subsubsection{High-fidelity model}

The high-fidelity model is the rectangular AC optimal power flow implemented by
% \texttt{ACPM} through \texttt{PowerModels.ACRPowerModel} and
% \texttt{PowerModels.build\_opf}.  For the quadratic generation costs used in
% this project, it is

\begin{equation}
\min_{\Xi^{\mathrm{AC}}}
\quad
\sum_{g\in G}
\left(C_g^{(2)}p_g^2+C_g^{(1)}p_g+C_g^{(0)}\right)
\label{eq:acpm-objective}
\end{equation}

s.t.\\

\begin{equation}
\sum_{g\in G_i}p_g-P_i^D
-G_i^{sh} c_{ii}
=
\sum_{(i,j)\in E_i}p_{ij}
+\sum_{(j,i)\in E_i}p_{ji}
\quad \forall i\in B
\label{eq:acpm-active-balance}
\end{equation}

\begin{equation}
\sum_{g\in G_i}q_g-Q_i^D
+B_i^{sh} c_{ii}
=
\sum_{(i,j)\in E_i}q_{ij}
+\sum_{(j,i)\in E_i}q_{ji}
\quad \forall i\in B
\label{eq:acpm-reactive-balance}
\end{equation}

The AC branch flows are

\begin{equation}
p_{ij}
=\frac{G_{ij}}{T_{ij}^2}c_{ii}
-\frac{G_{ij}}{T_{ij}}c_{ij}
+\frac{B_{ij}}{T_{ij}}s_{ij}
\quad \forall (i,j)\in E
\label{eq:acpm-ohms-from-p}
\end{equation}

\begin{equation}
p_{ji}
=G_{ij}c_{jj}
-\frac{G_{ij}}{T_{ij}}c_{ij}
-\frac{B_{ij}}{T_{ij}}s_{ij}
\quad \forall (i,j)\in E
\label{eq:acpm-ohms-to-p}
\end{equation}

\begin{equation}
q_{ij}
=-\frac{B_{ij}+B_{ij}^{\mathrm{sh}}}{T_{ij}^2}c_{ii}
+\frac{G_{ij}}{T_{ij}}s_{ij}
+\frac{B_{ij}}{T_{ij}}c_{ij}
\quad \forall (i,j)\in E
\label{eq:acpm-ohms-from-q}
\end{equation}

\begin{equation}
q_{ji}
=-(B_{ij}+B_{ij}^{\mathrm{sh}})c_{jj}
-\frac{G_{ij}}{T_{ij}}s_{ij}
+\frac{B_{ij}}{T_{ij}}c_{ij}
\quad \forall (i,j)\in E
\label{eq:acpm-ohms-to-q}
\end{equation}

\begin{equation}
-\tan(\Delta_{ij}^U)c_{ij}
\le s_{ij}\le
-\tan(\Delta_{ij}^L)c_{ij}
\quad \forall (i,j)\in E
\label{eq:acpm-angle-difference}
\end{equation}

\begin{equation}
p_{ij}^2+q_{ij}^2\le(S_{ij}^U)^2
\quad \forall (i,j),(j,i)\in E
\label{eq:acpm-thermal-limits}
\end{equation}

\begin{equation}
(V_i^L)^2
\le c_{ii}
\le(V_i^U)^2
\quad \forall i\in B
\label{eq:acpm-voltage-magnitude-bounds}
\end{equation}

\begin{equation}
-V_i^U\le v_i^r\le V_i^U,
\qquad
-V_i^U\le v_i^i\le V_i^U
\quad \forall i\in B
\label{eq:acpm-voltage-component-bounds}
\end{equation}

\begin{equation}
v_i^i=0
\quad \forall i\in R
\label{eq:acpm-reference}
\end{equation}

\begin{equation}
P_g^L\le p_g\le P_g^U,
\qquad
Q_g^L\le q_g\le Q_g^U
\quad \forall g\in G.
\label{eq:acpm-generator-bounds}
\end{equation}

\subsubsection{Low-fidelity model}

The low-fidelity model is the learned DC optimal power flow implemented by
% \texttt{DCPMLayer}.  The GNN supplies $K_i^{\mathrm{GNN}}$, which modifies the
% active demand appearing in the bus balance.  Reactive power, voltage
% magnitudes, AC losses, transformer taps, and phase shifts are omitted.

\begin{equation}
\min_{\Xi^{\mathrm{DC}}}
\quad
\sum_{g\in G}
\left(C_g^{(2)}p_g^2+C_g^{(1)}p_g+C_g^{(0)}\right)
+10C^{\max,1}\sum_{i\in B}z_i
\label{eq:dc-layer-objective}
\end{equation}

s.t.\\

\begin{equation}
\sum_{(i,j)\in E_i}p_{ij}
-\sum_{(j,i)\in E_i}p_{ji}
=
\sum_{g\in G_i}p_g
-P_i^D \bm{\alpha_i^{\mathrm{GNN}}}-G_i^{sh}+z_i
\quad \forall i\in B
\label{eq:dc-layer-balance}
\end{equation}

\begin{equation}
p_{ij}
=B_{ij}(\theta_i-\theta_j)
\quad \forall (i,j)\in E
\label{eq:dc-layer-flow}
\end{equation}

\begin{equation}
\Delta_{ij}^L
\le\theta_i-\theta_j
\le\Delta_{ij}^U
\quad \forall (i,j)\in E
\label{eq:dc-layer-angle-difference}
\end{equation}

\begin{equation}
-S_{ij}^U\le p_{ij}\le S_{ij}^U
\quad \forall (i,j)\in E
\label{eq:dc-layer-thermal-limit}
\end{equation}

\begin{equation}
\theta_i=0
\quad \forall i\in R
\label{eq:dc-layer-reference}
\end{equation}

\begin{equation}
P_g^L\le p_g\le P_g^U
\quad \forall g\in G
\label{eq:dc-layer-generator-bounds}
\end{equation}

\begin{equation}
z_i\ge0
\quad \forall i\in B.
\label{eq:dc-layer-slack}
\end{equation}

Let $p_g^*$ denote the generator component of an optimal low-fidelity
solution.  The parameter passed to the final high-fidelity layer is

\begin{equation}
\bm{\widehat{P}^{\mathrm{LF}}}_g:=p_g^*
\quad \forall g\in G.
\label{eq:dc-layer-solution}
\end{equation}

\subsubsection{High-fidelity model with slacks and the low-fidelity solution}

% The final layer, implemented by \texttt{ACPMLayer}, restores the complete AC
% physics and permits active generation to move away from the low-fidelity
% dispatch through positive and negative slacks.

\begin{equation}
\begin{aligned}
\min_{\Xi^{\mathrm{AC-S}}}
\quad
&\sum_{g\in G}
\left(C_g^{(2)}p_g^2+C_g^{(1)}p_g+C_g^{(0)}\right)\\
&+\sum_{g\in G}
\left[
C_g^{(2)}(s_g^++s_g^-)^2
+C_g^{(1)}(s_g^++s_g^-)
\right]
\end{aligned}
\label{eq:ac-layer-objective}
\end{equation}

s.t.\\

\begin{equation}
p_g=\bm{\widehat{P}^{\mathrm{LF}}}_g+s_g^+-s_g^-
\quad \forall g\in G
\label{eq:ac-layer-link}
\end{equation}

\begin{equation}
s_g^+\ge0,
\qquad
s_g^-\ge0
\quad \forall g\in G,
\label{eq:ac-layer-slacks}
\end{equation}

\begin{equation}
c_{ii}=(v_i^r)^2+(v_i^i)^2
\quad \forall i\in B
\label{eq:ac-layer-voltage-product-diagonal}
\end{equation}

\begin{equation}
c_{ij}=v_i^r v_j^r+v_i^i v_j^i
\quad \forall (i,j)\in E
\label{eq:ac-layer-voltage-product-real}
\end{equation}

\begin{equation}
s_{ij}=v_i^r v_j^i-v_j^r v_i^i
\quad \forall (i,j)\in E
\label{eq:ac-layer-voltage-product-imaginary}
\end{equation}

\begin{equation}
\sum_{g\in G_i}p_g-P_i^D-G_i^{sh} c_{ii}
=
\sum_{(i,j)\in E_i}p_{ij}
+\sum_{(j,i)\in E_i}p_{ji}
\quad \forall i\in B
\label{eq:ac-layer-active-balance}
\end{equation}

\begin{equation}
\sum_{g\in G_i}q_g-Q_i^D+B_i^{sh} c_{ii}
=
\sum_{(i,j)\in E_i}q_{ij}
+\sum_{(j,i)\in E_i}q_{ji}
\quad \forall i\in B
\label{eq:ac-layer-reactive-balance}
\end{equation}

\begin{equation}
p_{ij}
=\frac{G_{ij}}{T_{ij}^2}c_{ii}
-\frac{G_{ij}}{T_{ij}}c_{ij}
+\frac{B_{ij}}{T_{ij}}s_{ij}
\quad \forall (i,j)\in E
\label{eq:ac-layer-ohms-from-p}
\end{equation}

\begin{equation}
p_{ji}
=G_{ij}c_{jj}
-\frac{G_{ij}}{T_{ij}}c_{ij}
-\frac{B_{ij}}{T_{ij}}s_{ij}
\quad \forall (i,j)\in E
\label{eq:ac-layer-ohms-to-p}
\end{equation}

\begin{equation}
q_{ij}
=-\frac{B_{ij}+B_{ij}^{\mathrm{sh}}}{T_{ij}^2}c_{ii}
+\frac{G_{ij}}{T_{ij}}s_{ij}
+\frac{B_{ij}}{T_{ij}}c_{ij}
\quad \forall (i,j)\in E
\label{eq:ac-layer-ohms-from-q}
\end{equation}

\begin{equation}
q_{ji}
=-(B_{ij}+B_{ij}^{\mathrm{sh}})c_{jj}
-\frac{G_{ij}}{T_{ij}}s_{ij}
+\frac{B_{ij}}{T_{ij}}c_{ij}
\quad \forall (i,j)\in E
\label{eq:ac-layer-ohms-to-q}
\end{equation}

\begin{equation}
-\tan(\Delta_{ij}^U)c_{ij}
\le s_{ij}\le
-\tan(\Delta_{ij}^L)c_{ij}
\quad \forall (i,j)\in E
\label{eq:ac-layer-angle-difference}
\end{equation}

\begin{equation}
p_{ij}^2+q_{ij}^2\le(S_{ij}^U)^2
\quad \forall (i,j),(j,i)\in E
\label{eq:ac-layer-thermal-limits}
\end{equation}

\begin{equation}
(V_i^L)^2\le c_{ii}\le(V_i^U)^2
\quad \forall i\in B
\label{eq:ac-layer-voltage-magnitude-bounds}
\end{equation}

\begin{equation}
-V_i^U\le v_i^r\le V_i^U,
\qquad
-V_i^U\le v_i^i\le V_i^U
\quad \forall i\in B
\label{eq:ac-layer-voltage-component-bounds}
\end{equation}

\begin{equation}
v_i^i=0
\quad \forall i\in R
\label{eq:ac-layer-reference}
\end{equation}

\begin{equation}
P_g^L\le p_g\le P_g^U,
\qquad
Q_g^L\le q_g\le Q_g^U
\quad \forall g\in G.
\label{eq:ac-layer-generator-bounds}
\end{equation}

\section{Mathematical Formulation of Gas Flow Problem}

\subsection{Notation}

\subsubsection{Indices}

\begin{longtable}{@{}p{0.34\linewidth}p{0.60\linewidth}@{}}
\toprule
\textbf{Index} & \textbf{Description} \\
\midrule
$i,j$ & Junctions \\
$k$ & Pipes \\
$c$ & Compressors \\
$r$ & Receipts (gas supplies) \\
$d$ & Deliveries (gas withdrawals) \\
$v$ & Vertices in the polyhedral pipe relaxation \\
\bottomrule
\end{longtable}

\subsubsection{Sets}

\begin{longtable}{@{}p{0.34\linewidth}p{0.60\linewidth}@{}}
\toprule
\textbf{Set} & \textbf{Description} \\
\midrule
$J$ & Set of junctions \\
$P$ & Set of pipes \\
$C$ & Set of compressors \\
$R$ & Set of dispatchable receipts \\
$D$ & Set of nondispatchable deliveries \\
$R_i$ & Set of receipts connected to junction $i$ \\
$D_i$ & Set of deliveries connected to junction $i$ \\
$P_i^{+}$ & Set of pipes whose stored orientation leaves junction $i$ \\
$P_i^{-}$ & Set of pipes whose stored orientation enters junction $i$ \\
$C_i^{+}$ & Set of compressors whose stored orientation leaves junction $i$ \\
$C_i^{-}$ & Set of compressors whose stored orientation enters junction $i$ \\
$V_k$ & Set of vertices in the polyhedral relaxation for pipe $k$ \\
\bottomrule
\end{longtable}

\subsubsection{Parameters}

\begin{longtable}{@{}p{0.34\linewidth}p{0.60\linewidth}@{}}
\toprule
\textbf{Parameter} & \textbf{Description} \\
\midrule
\endfirsthead

\toprule
\textbf{Parameter} & \textbf{Description} \\
\midrule
\endhead

$F_k$ & From-junction of pipe $k$ \\
$T_k$ & To-junction of pipe $k$ \\
$F_c$ & From-junction of compressor $c$ \\
$T_c$ & To-junction of compressor $c$ \\
$P_i^L$ & Lower pressure bound at junction $i$ \\
$P_i^U$ & Upper pressure bound at junction $i$ \\
$\Pi_i^L=(P_i^L)^2$ & Lower squared-pressure bound at junction $i$ \\
$\Pi_i^U=(P_i^U)^2$ & Upper squared-pressure bound at junction $i$ \\
$Q_k^L$ & Lower mass-flow bound of pipe $k$ \\
$Q_k^U$ & Upper mass-flow bound of pipe $k$ \\
$H_c^L$ & Lower mass-flow bound of compressor $c$ \\
$H_c^U$ & Upper mass-flow bound of compressor $c$ \\
$\Delta\Pi_k^L$ & Lower squared-pressure-drop bound of pipe $k$ \\
$\Delta\Pi_k^U$ & Upper squared-pressure-drop bound of pipe $k$ \\
$W_k$ & Weymouth coefficient of pipe $k$ \\
$R_c^L$ & Minimum pressure ratio of compressor $c$ \\
$R_c^U$ & Maximum pressure ratio of compressor $c$ \\
$G_r^L$ & Minimum injection of receipt $r$ \\
$G_r^U$ & Maximum injection of receipt $r$ \\
$L_d$ & Fixed withdrawal of delivery $d$ for the selected sample \\
$K_r$ & Offer price of receipt $r$ \\
$(X_{kv},Y_{kv})$ & Vertex $v$ of the relaxation of $y=x|x|$ for pipe $k$ \\
$\alpha_k^{\mathrm{GNN}}$ & GNN coefficient for pipe $k$ \\
$\beta_i^{\mathrm{GNN}}$ & GNN coefficient for junction $i$ \\
$\widehat{G}^{\mathrm{LF}}_r$ & Receipt injection obtained from the low-fidelity model and used as a parameter in the final high-fidelity model \\
\bottomrule
\end{longtable}

\subsubsection{Variables}

\begin{longtable}{@{}p{0.34\linewidth}p{0.60\linewidth}@{}}
\toprule
\textbf{Variable} & \textbf{Description} \\
\midrule
\endfirsthead

\toprule
\textbf{Variable} & \textbf{Description} \\
\midrule
\endhead

$g_r$ & Injection at receipt $r$ \\
$f_k$ & Signed mass flow through pipe $k$ \\
$h_c$ & Signed mass flow through compressor $c$ \\
$\pi_i$ & Squared pressure at junction $i$ \\
$\rho_c$ & Squared pressure-ratio variable of compressor $c$ \\
$q_k$ & Low-fidelity auxiliary for $f_k|f_k|$ \\
$\lambda_{kv}$ & Convex-combination weight for pipe-relaxation vertex $v$ \\
$\delta_c$ & Compressor pressure-difference auxiliary \\
$u_c$ & Low-fidelity auxiliary for $h_c\delta_c$ \\
$s_r^+$ & Positive deviation from the low-fidelity receipt solution \\
$s_r^-$ & Negative deviation from the low-fidelity receipt solution \\
\bottomrule
\end{longtable}

\subsubsection{McCormick relaxation}

For variables $x\in[X^L,X^U]$, $y\in[Y^L,Y^U]$, and product surrogate $z$, the notation
$\operatorname{MC}(x,y,z)$ represents the following four inequalities:

\begin{equation}
z \ge X^L y + Y^L x - X^L Y^L
\end{equation}

\begin{equation}
z \ge X^U y + Y^U x - X^U Y^U
\end{equation}

\begin{equation}
z \le X^L y + Y^U x - X^L Y^U
\end{equation}

\begin{equation}
z \le X^U y + Y^L x - X^U Y^L
\end{equation}

\subsection{Formulation}

\subsubsection{High-fidelity model}

% The high-fidelity model is the Weymouth-pressure model instantiated by
\texttt{WPGasModel} and \texttt{GasModels.build\_ogf\_comp\_power\_unc}.

\begin{equation}
\min_{g,f,h,\pi,\rho}
\left\{
\sum_{r\in R} K_r g_r
\right\}
\label{eq:hf-objective}
\end{equation}

s.t.\\

\begin{equation}
\sum_{r\in R_i}g_r
-\sum_{d\in D_i}L_d
=
\sum_{k\in P_i^+}f_k-
\sum_{k\in P_i^-}f_k+
\sum_{c\in C_i^+}h_c-
\sum_{c\in C_i^-}h_c
\quad \forall i\in J
\label{eq:hf-balance}
\end{equation}

\begin{equation}
G_r^L \le g_r \le G_r^U
\quad \forall r\in R
\label{eq:hf-receipt-bounds}
\end{equation}

\begin{equation}
\Pi_i^L \le \pi_i \le \Pi_i^U
\quad \forall i\in J
\label{eq:hf-pressure-bounds}
\end{equation}

\begin{equation}
Q_k^L \le f_k \le Q_k^U
\quad \forall k\in P
\label{eq:hf-pipe-flow-bounds}
\end{equation}

\begin{equation}
\Delta\Pi_k^L
\le \pi_{F_k}-\pi_{T_k}
\le \Delta\Pi_k^U
\quad \forall k\in P
\label{eq:hf-pressure-drop}
\end{equation}

\begin{equation}
W_k\left(\pi_{F_k}-\pi_{T_k}\right)
= f_k|f_k|
\quad \forall k\in P
\label{eq:hf-weymouth}
\end{equation}

\begin{equation}
H_c^L \le h_c \le H_c^U
\quad \forall c\in C
\label{eq:hf-compressor-flow-bounds}
\end{equation}

\begin{equation}
(R_c^L)^2\pi_{F_c}
\le \pi_{T_c}
\le (R_c^U)^2\pi_{F_c}
\quad \forall c\in C
\label{eq:hf-compressor-pressure}
\end{equation}

\begin{equation}
h_c\left(\pi_{F_c}-\pi_{T_c}\right) \le 0
\quad \forall c\in C
\label{eq:hf-compressor-direction}
\end{equation}

\begin{equation}
\pi_{T_c}=\rho_c\pi_{F_c}
\quad \forall c\in C
\label{eq:hf-compressor-product}
\end{equation}

\begin{equation}
(R_c^L)^2 \le \rho_c \le (R_c^U)^2
\quad \forall c\in C
\label{eq:hf-ratio-bounds}
\end{equation}

\subsubsection{Low-fidelity model}

The low-fidelity model uses the same network balances and bounds, replaces the nonlinear pipe and
compressor products by relaxations, and adds the GNN-generated coefficients $A_k$ and $B_i$ to the
objective.

\begin{equation}
\begin{aligned}
\min_{g,f,h,\pi,\rho,q,\lambda,\delta,u}
\left\{
\sum_{r\in R} K_r g_r
+\sum_{k\in P} \bm{\alpha_k^{\mathrm{GNN}}} f_k^2
+\sum_{i\in J} \bm{\beta_i^{\mathrm{GNN}}}\pi_i^2
\right\}
\end{aligned}
\label{eq:lf-objective}
\end{equation}

s.t.\\

\begin{equation}
\sum_{r\in R_i}g_r
-\sum_{d\in D_i}L_d
=
\sum_{k\in P_i^+}f_k-
\sum_{k\in P_i^-}f_k+
\sum_{c\in C_i^+}h_c-
\sum_{c\in C_i^-}h_c
\quad \forall i\in J
\label{eq:lf-balance}
\end{equation}

\begin{equation}
G_r^L \le g_r \le G_r^U
\quad \forall r\in R
\end{equation}

\begin{equation}
\Pi_i^L \le \pi_i \le \Pi_i^U
\quad \forall i\in J
\end{equation}

\begin{equation}
Q_k^L \le f_k \le Q_k^U
\quad \forall k\in P
\end{equation}

\begin{equation}
\Delta\Pi_k^L
\le \pi_{F_k}-\pi_{T_k}
\le \Delta\Pi_k^U
\quad \forall k\in P
\end{equation}

\begin{equation}
W_k\left(\pi_{F_k}-\pi_{T_k}\right)=q_k
\quad \forall k\in P
\label{eq:lf-relaxed-weymouth}
\end{equation}

\begin{equation}
\sum_{v\in V_k}\lambda_{kv}=1
\quad \forall k\in P
\end{equation}

\begin{equation}
f_k=\sum_{v\in V_k}X_{kv}\lambda_{kv}
\quad \forall k\in P
\end{equation}

\begin{equation}
q_k=\sum_{v\in V_k}Y_{kv}\lambda_{kv}
\quad \forall k\in P
\end{equation}

\begin{equation}
0\le\lambda_{kv}\le1
\quad \forall k\in P,\;v\in V_k
\end{equation}

\begin{equation}
H_c^L \le h_c \le H_c^U
\quad \forall c\in C
\end{equation}

\begin{equation}
(R_c^L)^2\pi_{F_c}
\le \pi_{T_c}
\le (R_c^U)^2\pi_{F_c}
\quad \forall c\in C
\end{equation}

\begin{equation}
\delta_c=\pi_{F_c}-\pi_{T_c}
\quad \forall c\in C
\end{equation}

\begin{equation}
\Pi_{F_c}^L-\Pi_{T_c}^U
\le\delta_c\le
\Pi_{F_c}^U-\Pi_{T_c}^L
\quad \forall c\in C
\end{equation}

\begin{equation}
\operatorname{MC}(h_c,\delta_c,u_c),
\qquad u_c\le0
\quad \forall c\in C
\label{eq:lf-compressor-direction}
\end{equation}

\begin{equation}
\operatorname{MC}(\pi_{F_c},\rho_c,\pi_{T_c})
\quad \forall c\in C
\label{eq:lf-compressor-product}
\end{equation}

\begin{equation}
(R_c^L)^2 \le \rho_c \le (R_c^U)^2
\quad \forall c\in C
\end{equation}

Let $g_r^*$ denote the receipt component of an optimal low-fidelity solution.  The value passed to
the next layer is the parameter

\begin{equation}
\bm{\widehat{G}^{\mathrm{LF}}_r}:=g_r^*
\quad \forall r\in R.
\label{eq:lf-solution}
\end{equation}

\subsubsection{High-fidelity model with slacks and the low-fidelity solution: Conservative/Warm start}

The last optimization layer restores the exact high-fidelity physics and permits the receipt
injections to move away from the low-fidelity receipt solution through positive and negative
slacks.

\begin{equation}
\begin{aligned}
\min_{g,f,h,\pi,\rho,s^+,s^-}
\left\{
\sum_{r\in R}K_r g_r
+\sum_{r\in R}K_r\left(s_r^++s_r^-\right)^2
\right\}
\end{aligned}
\label{eq:hfs-objective}
\end{equation}

s.t.\\

\begin{equation}
g_r
=\bm{\widehat{G}^{\mathrm{LF}}_r}
+s_r^+-s_r^-
\quad \forall r\in R
\label{eq:hfs-link}
\end{equation}

\begin{equation}
s_r^+\ge0,
\qquad
s_r^-\ge0
\quad \forall r\in R
\label{eq:hfs-slacks}
\end{equation}

\begin{equation}
\sum_{r\in R_i}g_r
-\sum_{d\in D_i}L_d
=
\sum_{k\in P_i^+}f_k-
\sum_{k\in P_i^-}f_k+
\sum_{c\in C_i^+}h_c-
\sum_{c\in C_i^-}h_c
\quad \forall i\in J
\end{equation}

\begin{equation}
G_r^L \le g_r \le G_r^U
\quad \forall r\in R
\end{equation}

\begin{equation}
\Pi_i^L \le \pi_i \le \Pi_i^U
\quad \forall i\in J
\end{equation}

\begin{equation}
Q_k^L \le f_k \le Q_k^U
\quad \forall k\in P
\end{equation}

\begin{equation}
\Delta\Pi_k^L
\le \pi_{F_k}-\pi_{T_k}
\le \Delta\Pi_k^U
\quad \forall k\in P
\end{equation}

\begin{equation}
W_k\left(\pi_{F_k}-\pi_{T_k}\right)
= f_k|f_k|
\quad \forall k\in P
\end{equation}

\begin{equation}
H_c^L \le h_c \le H_c^U
\quad \forall c\in C
\end{equation}

\begin{equation}
(R_c^L)^2\pi_{F_c}
\le \pi_{T_c}
\le (R_c^U)^2\pi_{F_c}
\quad \forall c\in C
\end{equation}

\begin{equation}
h_c\left(\pi_{F_c}-\pi_{T_c}\right)\le0
\quad \forall c\in C
\end{equation}

\begin{equation}
\pi_{T_c}=\rho_c\pi_{F_c}
\quad \forall c\in C
\end{equation}

\begin{equation}
(R_c^L)^2 \le \rho_c \le (R_c^U)^2
\quad \forall c\in C
\end{equation}

\section{Mathematical Formulation of Pooling Problem}

\subsection{Notation}

\subsubsection{Indices}

\begin{longtable}{@{}p{0.34\linewidth}p{0.60\linewidth}@{}}
\toprule
\textbf{Index} & \textbf{Description} \\
\midrule
$i$ & Raw materials \\
$j$ & Products \\
$k$ & Qualities \\
$l$ & Pools \\
\bottomrule
\end{longtable}

\subsubsection{Sets}

\begin{longtable}{@{}p{0.34\linewidth}p{0.60\linewidth}@{}}
\toprule
\textbf{Set} & \textbf{Description} \\
\midrule
$I$ & Set of raw materials \\
$J$ & Set of products \\
$K$ & Set of qualities \\
$L$ & Set of pools \\
$T_X$ & Set of input-to-pool connections $(i,l)$ (pool feeds) \\
$T_Y$ & Set of pool-to-product connections $(l,j)$ \\
$T_Z$ & Set of direct input-to-product connections $(i,j)$ (direct feeds) \\
$T_V$ & Set of input--pool--product paths $(i,l,j)$ such that $(i,l)\in T_X$ and $(l,j)\in T_Y$ \\
\bottomrule
\end{longtable}

\subsubsection{Parameters}

\begin{longtable}{@{}p{0.34\linewidth}p{0.60\linewidth}@{}}
\toprule
\textbf{Parameter} & \textbf{Description} \\
\midrule
\endfirsthead

\toprule
\textbf{Parameter} & \textbf{Description} \\
\midrule
\endhead

$F_{ij}$ & Per-unit flow weight for direct arc $(i,j)$ \\
$F_{il}$ & Per-unit flow weight for input-to-pool arc $(i,l)$ \\
$F_{lj}$ & Per-unit flow weight for pool-to-product arc $(l,j)$ \\
$A_i^U$ & Upper bound on the availability of raw material $i$ \\
$P_{i,k}$ & Value of quality $k$ for raw material $i$ \\
$D_j^U$ & Upper bound on demand for product $j$ \\
$P^U_{j,k}$ & Upper bound on quality $k$ for product $j$ \\
$P^L_{j,k}$ & Lower bound on quality $k$ for product $j$ \\
$S_l$ & Capacity of pool $l$ \\
$S_{l,j}$ & Capacity of flow from pool $l$ to product $j$ \\
$S_{i,j}$ & Capacity of direct flow from raw material $i$ to product $j$ \\
$\alpha^{\mathrm{GNN}}_{ilj}$ & GNN coefficient for path-flow $x_{ilj}$, for $(i,l,j)\in T_V$ \\
$\beta^{\mathrm{GNN}}_{lj}$ & GNN coefficient for pool-to-product flow $y_{lj}$, for $(l,j)\in T_Y$ \\
$\sigma^{\mathrm{GNN}}_{ij}$ & GNN coefficient for direct-flow variable $z_{ij}$, for $(i,j)\in T_Z$ \\
$\eta^{\mathrm{GNN}}_{il}$ & GNN coefficient for pool-composition variable $q_{il}$, for $(i,l)\in T_X$ \\
\bottomrule
\end{longtable}

\subsubsection{Variables}

\begin{longtable}{@{}p{0.34\linewidth}p{0.60\linewidth}@{}}
\toprule
\textbf{Variable} & \textbf{Description} \\
\midrule
$y_{l,j}$ & Flow from pool $l$ to product $j$ \\
$z_{i,j}$ & Direct flow from raw material $i$ to product $j$ \\
$p_{l,k}$ & Level of quality $k$ in pool $l$ \\
$q_{i,l}$ & Proportion of flow from raw material $i$ to pool $l$ \\
$v_{i,l,j}$ & Flow from raw material $i$ to product $j$ through pool $l$ \\
\bottomrule
\end{longtable}

\subsection{Formulation}
\subsubsection{Full space model}

\begin{equation}
\min_{y,z,v,q}
\left\{
- \sum_{(i,j)\in T_Z} F_{ij} z_{ij}
- \sum_{(i,l,j)\in T_V} (F_{il} + F_{lj}) v_{ilj}
\right\}
\end{equation}

\text{s.t.}

\begin{equation}
v_{ilj} = q_{il} y_{lj} \quad \forall (i,l,j)\in T_V
\end{equation}

\begin{equation}
\sum_{i:\,(i,l)\in T_X} q_{il} = 1 \quad \forall l \in L
\end{equation}

\begin{equation}
\sum_{i:\,(i,j)\in T_Z} P_{i,k} z_{ij}
+ \sum_{(i,l):\,(i,l,j)\in T_V} P_{i,k} v_{ilj}
\le
P^U_{j,k}
\left(
\sum_{l:\,(l,j)\in T_Y} y_{lj}
+
\sum_{i:\,(i,j)\in T_Z} z_{ij}
\right)
\quad \forall j \in J,\; k \in K
\end{equation}

\begin{equation}
\sum_{i:\,(i,j)\in T_Z} P_{i,k} z_{ij}
+ \sum_{(i,l):\,(i,l,j)\in T_V} P_{i,k} v_{ilj}
\ge
P^L_{j,k}
\left(
\sum_{l:\,(l,j)\in T_Y} y_{lj}
+
\sum_{i:\,(i,j)\in T_Z} z_{ij}
\right)
\quad \forall j \in J,\; k \in K
\end{equation}

\begin{equation}
\sum_{j:\,(i,j)\in T_Z} z_{ij}
+
\sum_{(l,j):\,(i,l,j)\in T_V} v_{ilj}
\le A_i
\quad \forall i \in I
\end{equation}

\begin{equation}
\sum_{j:\,(l,j)\in T_Y} y_{lj}
\le S_l
\quad \forall l \in L
\end{equation}

\begin{equation}
\sum_{i:\,(i,j)\in T_Z} z_{ij}
+
\sum_{(i,l):\,(i,l,j)\in T_V} v_{ilj}
\le D_j
\quad \forall j \in J
\end{equation}

\begin{equation}
y_{lj} \le S_{l,j} \quad \forall (l,j)\in T_Y, 
\qquad
z_{ij} \le S_{i,j} \quad \forall (i,j)\in T_Z
\end{equation}

\begin{equation}
\sum_{i:\,(i,l,j)\in T_V} v_{ilj}
= y_{lj}
\quad \forall (l,j)\in T_Y
\end{equation}

\begin{equation}
\sum_{j:\,(i,l,j)\in T_V} v_{ilj}
\le S_l\, q_{il}
\quad \forall (i,l)\in T_X
\end{equation}

\begin{equation}
q_{il} \ge 0,\;
y_{lj} \ge 0,\;
z_{ij} \ge 0,\;
v_{ilj} \ge 0
\end{equation}
\subsubsection{Low fidelity model}
\begin{equation}
\begin{aligned}
\min_{z, v, q, y} \{
& \sum_{(i,l,j)\in ILJ} \left\{0.1\, v_{ilj}^2 - \bm{\alpha_{i,l,j}^{GNN}}(F_{il} + F_{ij})v_{ilj} \right\} \\
& + \sum_{(i,j)\in IJ} \left\{0.1\, z_{i,j}^2 - \bm{\beta_{i,j}^{GNN}} F_{ij} z_{i,j} \right\} \\
& + \sum_{(i,l) \in IL} \left\{0.1\, q_{i,l}^2 - \bm{\sigma_{i,l}^{GNN}} q_{i,l} \right\} \\
& + \sum_{(l,j) \in LJ} \left\{0.1\, y_{l,j}^2 - \bm{\eta_{l,j}^{GNN}} y_{l,j} \right\}
\}
\end{aligned}
\end{equation}
s.t.\\
\begin{equation}
c_{l,j} =
\begin{cases}
F_{l,j}, & \text{if } (l,j) \in T_Y\\
S_l, & \text{otherwise}
\end{cases}
\end{equation}
\begin{equation}
x_{i,l,j} \: \ge y_{l,j} + c_{l,j}\, q_{i,l} - c_{l,j} \quad \forall (i,l,j)\in T_v 
\end{equation}
\begin{equation}
x_{i,l,j} \: \leq y_{l,j} \quad \forall (i,l,j)\in T_v 
\end{equation}
\begin{equation}
x_{i,l,j} \: \le c_{l,j}\, q_{i,l} \quad \forall (i,l,j)\in T_v
\end{equation}
\begin{equation}
\sum_{(i,l) \in T_x} q_{il} = 1 \quad \forall l \in L
\end{equation}
\begin{equation}
\sum_{i:\,(i,j)\in T_Z} P_{i,k} z_{ij}
+ \sum_{(i,l):\,(i,l,j)\in T_V} P_{i,k} v_{ilj}
\le
P^U_{j,k}
\left(
\sum_{l:\,(l,j)\in T_Y} y_{lj}
+
\sum_{i:\,(i,j)\in T_Z} z_{ij}
\right)
\quad \forall j \in J,\; k \in K
\end{equation}

\begin{equation}
\sum_{i:\,(i,j)\in T_Z} P_{i,k} z_{ij}
+ \sum_{(i,l):\,(i,l,j)\in T_V} P_{i,k} v_{ilj}
\ge
P^L_{j,k}
\left(
\sum_{l:\,(l,j)\in T_Y} y_{lj}
+
\sum_{i:\,(i,j)\in T_Z} z_{ij}
\right)
\quad \forall j \in J,\; k \in K
\end{equation}

\begin{equation}
\sum_{j:\,(i,j)\in T_Z} z_{ij}
+
\sum_{(l,j):\,(i,l,j)\in T_V} v_{ilj}
\le A_i
\quad \forall i \in I
\end{equation}

\begin{equation}
\sum_{j:\,(l,j)\in T_Y} y_{lj}
\le S_l
\quad \forall l \in L
\end{equation}

\begin{equation}
\sum_{i:\,(i,j)\in T_Z} z_{ij}
+
\sum_{(i,l):\,(i,l,j)\in T_V} v_{ilj}
\le D_j
\quad \forall j \in J
\end{equation}

\begin{equation}
y_{lj} \le S_{l,j} \quad \forall (l,j)\in T_Y, 
\qquad
z_{ij} \le S_{i,j} \quad \forall (i,j)\in T_Z
\end{equation}

\begin{equation}
\sum_{i:\,(i,l,j)\in T_V} v_{ilj}
= y_{lj}
\quad \forall (l,j)\in T_Y
\end{equation}

\begin{equation}
\sum_{j:\,(i,l,j)\in T_V} v_{ilj}
\le S_l\, q_{il}
\quad \forall (i,l)\in T_X
\end{equation}

\begin{equation}
q_{il} \ge 0,\;
y_{lj} \ge 0,\;
z_{ij} \ge 0,\;
v_{ilj} \ge 0
\end{equation}
\subsubsection{High-fidelity model with slacks and q variable from low-fidelity model : Conservative/Warm start}
\begin{equation}
\begin{aligned}
\min_{y,z,v,s} \; - \sum_{(i,j)\in T_Z} F_{ij} \, z_{i,j}
- \sum_{(i,l,j)\in T_v} (F_{il} + F_{lj}) \, v_{i,l,j}
+ \sum_{(i,l,j)\in T_v} \max\big(F_{il} + F_{lj},\, 0\big)\,\big(s^+_{i,l,j} + s^-_{i,l,j}\big)
\end{aligned}
\end{equation}

s.t.\\
\begin{equation}
\begin{cases}
v_{i,l,j} = \bm{q^*_{i,l}}\, y_{l,j},  \quad s^{+}_{i,l,j} + s^{-}_{i,l,j} = 0
& \text{if } F_{il} + F_{lj} \le 0, \\[6pt]
v_{i,l,j} + s^{+}_{i,l,j} - s^{-}_{i,l,j} = \bm{q^*_{i,l}}\, y_{l,j} 
& \text{otherwise}
\end{cases}
\quad \forall (i,l,j) \in T_v
\end{equation}
\begin{equation}
\sum_{j:\,(i,l,j)\in T_v} v_{i,l,j} \le S_l\, \bm{q^*_{i,l}}
\quad \forall (i,l) \in T_X
\end{equation}
\begin{equation}
\begin{aligned}
s^{+}_{i,l,j} \:\ge 0 \quad \forall (i,l,j)\in T_v, \quad
s^{-}_{i,l,j} \:\ge 0 \quad \forall (i,l,j)\in T_v
\end{aligned}
\end{equation}
\begin{equation}
\sum_{i:\,(i,j)\in T_Z} P_{i,k} z_{ij}
+ \sum_{(i,l):\,(i,l,j)\in T_V} P_{i,k} v_{ilj}
\le
P^U_{j,k}
\left(
\sum_{l:\,(l,j)\in T_Y} y_{lj}
+
\sum_{i:\,(i,j)\in T_Z} z_{ij}
\right)
\quad \forall j \in J,\; k \in K
\end{equation}

\begin{equation}
\sum_{i:\,(i,j)\in T_Z} P_{i,k} z_{ij}
+ \sum_{(i,l):\,(i,l,j)\in T_V} P_{i,k} v_{ilj}
\ge
P^L_{j,k}
\left(
\sum_{l:\,(l,j)\in T_Y} y_{lj}
+
\sum_{i:\,(i,j)\in T_Z} z_{ij}
\right)
\quad \forall j \in J,\; k \in K
\end{equation}

\begin{equation}
\sum_{j:\,(i,j)\in T_Z} z_{ij}
+
\sum_{(l,j):\,(i,l,j)\in T_V} v_{ilj}
\le A_i
\quad \forall i \in I
\end{equation}

\begin{equation}
\sum_{j:\,(l,j)\in T_Y} y_{lj}
\le S_l
\quad \forall l \in L
\end{equation}

\begin{equation}
\sum_{i:\,(i,j)\in T_Z} z_{ij}
+
\sum_{(i,l):\,(i,l,j)\in T_V} v_{ilj}
\le D_j
\quad \forall j \in J
\end{equation}

\begin{equation}
y_{lj} \le S_{l,j} \quad \forall (l,j)\in T_Y, 
\qquad
z_{ij} \le S_{i,j} \quad \forall (i,j)\in T_Z
\end{equation}

\begin{equation}
\sum_{j:\,(i,l,j)\in T_V} v_{ilj}
\le S_l\, q_{il}
\quad \forall (i,l)\in T_X
\end{equation}

\begin{equation}
q_{il} \ge 0,\;
y_{lj} \ge 0,\;
z_{ij} \ge 0,\;
v_{ilj} \ge 0
\end{equation}

\subsubsection{High-fidelity model with slacks and q variable from low-fidelity model : Aggresive/final}
\begin{equation}
\begin{aligned}
\min_{y,z,v,s} \; - \sum_{(i,j)\in T_Z} F_{ij} \, z_{i,j}
- \sum_{(i,l,j)\in T_v} (F_{il} + F_{lj}) \, v_{i,l,j}
+\max_{(i,l,j)\in T_v} \big( F_{il} + F_{lj} \big)\sum_{(i,l,j)\in T_v} \,\big(s^+_{i,l,j} + s^-_{i,l,j}\big)
\end{aligned}
\end{equation}

s.t.\\
\begin{equation}
\begin{cases}
v_{i,l,j} = \bm{q^*_{i,l}}\, y_{l,j},  \quad s^{+}_{i,l,j} + s^{-}_{i,l,j} = 0
& \text{if } F_{il} + F_{lj} \le 0, \\[6pt]
v_{i,l,j} + s^{+}_{i,l,j} - s^{-}_{i,l,j} = \bm{q^*_{i,l}}\, y_{l,j} 
& \text{otherwise}
\end{cases}
\quad \forall (i,l,j) \in T_v
\end{equation}
\begin{equation}
\sum_{j:\,(i,l,j)\in T_v} v_{i,l,j} \le S_l\, \bm{q^*_{i,l}}
\quad \forall (i,l) \in T_X
\end{equation}
\begin{equation}
\begin{aligned}
s^{+}_{i,l,j} \:\ge 0 \quad \forall (i,l,j)\in T_v, \quad
s^{-}_{i,l,j} \:\ge 0 \quad \forall (i,l,j)\in T_v
\end{aligned}
\end{equation}
\begin{equation}
\sum_{i:\,(i,j)\in T_Z} P_{i,k} z_{ij}
+ \sum_{(i,l):\,(i,l,j)\in T_V} P_{i,k} v_{ilj}
\le
P^U_{j,k}
\left(
\sum_{l:\,(l,j)\in T_Y} y_{lj}
+
\sum_{i:\,(i,j)\in T_Z} z_{ij}
\right)
\quad \forall j \in J,\; k \in K
\end{equation}

\begin{equation}
\sum_{i:\,(i,j)\in T_Z} P_{i,k} z_{ij}
+ \sum_{(i,l):\,(i,l,j)\in T_V} P_{i,k} v_{ilj}
\ge
P^L_{j,k}
\left(
\sum_{l:\,(l,j)\in T_Y} y_{lj}
+
\sum_{i:\,(i,j)\in T_Z} z_{ij}
\right)
\quad \forall j \in J,\; k \in K
\end{equation}

\begin{equation}
\sum_{j:\,(i,j)\in T_Z} z_{ij}
+
\sum_{(l,j):\,(i,l,j)\in T_V} v_{ilj}
\le A_i
\quad \forall i \in I
\end{equation}

\begin{equation}
\sum_{j:\,(l,j)\in T_Y} y_{lj}
\le S_l
\quad \forall l \in L
\end{equation}

\begin{equation}
\sum_{i:\,(i,j)\in T_Z} z_{ij}
+
\sum_{(i,l):\,(i,l,j)\in T_V} v_{ilj}
\le D_j
\quad \forall j \in J
\end{equation}

\begin{equation}
y_{lj} \le S_{l,j} \quad \forall (l,j)\in T_Y, 
\qquad
z_{ij} \le S_{i,j} \quad \forall (i,j)\in T_Z
\end{equation}

\begin{equation}
\sum_{j:\,(i,l,j)\in T_V} v_{ilj}
\le S_l\, q_{il}
\quad \forall (i,l)\in T_X
\end{equation}

\begin{equation}
q_{il} \ge 0,\;
y_{lj} \ge 0,\;
z_{ij} \ge 0,\;
v_{ilj} \ge 0
\end{equation}

% \end{document}